\documentclass[a4paper, 12pt]{amsart}
\usepackage{amssymb,amsbsy,amsmath,amsfonts,amssymb,amscd}
\usepackage{latexsym}
\usepackage{graphics}
\usepackage{color}
\usepackage{marginnote}
\renewcommand{\Sigma}{\Gamma}

\newcommand\Oh{{\mathcal O}}

\newcommand\sH{{\mathcal H}}
\newcommand\sF{{\mathcal F}}
\newcommand\sG{{\mathcal G}}
\newcommand\sI{{\mathcal I}}

\newcommand\sL{{\mathcal L}}
\newcommand\sB{{\mathcal B}}
\newcommand\sQ{{\mathcal Q}}
\newcommand\sN{{\mathcal N}}

\newcommand\sK{{\mathcal K}}
\newcommand\LL{{\mathbb L}}
                \newcommand\sM{{\mathcal M}}

\newcommand\om{\omega}
\newcommand\Om{\Omega}
\newcommand{\ot}{\otimes}
\newcommand{\op}{\oplus}
\newcommand\la{\lambda}

\newcommand\Ga{\Sigma}

\newcommand\de{\delta}

\DeclareMathOperator{\spec}{Spec}
\DeclareMathOperator{\sing}{Sing}
\DeclareMathOperator{\coker}{coker}

\newcommand{\CC}{\ensuremath{\mathbb{C}}}

\newcommand{\ZZ}{\ensuremath{\mathbb{Z}}}
\newcommand{\QQ}{\ensuremath{\mathbb{Q}}}

\newcommand{\MMM}{\ensuremath{\mathcal{M}}}

\newcommand{\hol}{\ensuremath{\mathcal{O}}}

\newcommand{\PP}{\ensuremath{\mathbb{P}}}

\newcommand{\ra}{\ensuremath{\rightarrow}}

\def\eea{\end{eqnarray*}}
\def\bea{\begin{eqnarray*}}

\newcommand\dual{\mathrel{\raise3pt\hbox{$\underline{\mathrm{\thinspace d
\thinspace}}$}}}

\newcommand\qe{\ifhmode\unskip\nobreak\fi\quad $\Box$}

\def\BOX{\hfill\lower.5\baselineskip\hbox{$\Box$}}

\newtheorem{theo}{Theorem}[section]
\newtheorem{remarkk}[theo]{Remark}
\newenvironment{rem}{\begin{remarkk}\rm}{\end{remarkk}}

\newtheorem{defin}[theo]{Definition}

\newtheorem{prop}[theo] {Proposition}
\newtheorem{cor}[theo]{Corollary}
\newtheorem{lemma}[theo]{Lemma}
\newtheorem{example}[theo]{Example}

\newcommand{\sA}{\ensuremath{\mathcal{A}}}
\newcommand{\sR}{\ensuremath{\mathcal{R}}}

\newcommand{\sV}{\ensuremath{\mathcal{V}}}

\newcommand{\bZ}{{\mathbb Z}}
\newcommand{\bC}{{\mathbb C}}

\newcommand{\ul}{\underline}

\renewcommand{\a}{\alpha}
\renewcommand{\b}{\beta}

\renewcommand{\d}{\delta}
\newcommand{\D}{\Delta}
\newcommand{\e}{\varepsilon}
\newcommand{\we}{\wedge}

\newcommand{\z}{\zeta}

\newcommand{\s}{\sigma}

\newcommand{\cutoff}[1]{}

\begin{document}

\title[Dihedral Galois covers]{Dihedral Galois  covers of algebraic varieties and 
the simple cases}
\author{Fabrizio Catanese, Fabio Perroni}

\thanks{
The present work took place in the framework of the DFG
Forschergruppe 790 `Classification of algebraic
surfaces and compact complex manifolds' and of  
the ERC-2013-Advanced Grant - 340258- TADMICAMT. The second author was also
supported by the grant FRA 2015 of the University of Trieste,
PRIN ``Geometria delle Variet\`{a} Algebriche" and  INDAM}


\maketitle

\begin{center}
{\it Dedicated to Ugo Bruzzo on the occasion of his 60-th birthday.}
\end{center}

\begin{abstract}
In this article we investigate the algebra and geometry of 
dihedral covers of smooth algebraic varieties.
To this aim we first describe the Weil divisors and the Picard group of divisorial sheaves
on  normal double covers. Then we provide a structure theorem for dihedral covers, 
that is, given a smooth variety $Y$, we describe the algebraic ``building data"  on $Y$
which are equivalent to the existence  of 
such covers $\pi \colon X \to Y$. We introduce then two special 
very explicit classes of dihedral covers:
the simple and the almost simple dihedral covers, and we determine their
basic invariants. For the simple dihedral covers we also determine  their natural deformations. 
In the last section we give an application to fundamental groups. 
\end{abstract}

\tableofcontents


\section{Introduction}

A main issue in the classification theory of algebraic varieties is the construction of interesting and illuminating examples.
For instance in the book of Enriques \cite{enriques} one can see that a recurrent method is the one
of considering the minimal resolution $S$ of double covers $f : X \ra Y$;  these  are called `piani doppi', double planes,  when $Y = \PP^2$.

Burniat \cite{burniat} considered more general bidouble covers, i.e., Galois covers $f : X \ra Y$ with group $G = (\ZZ/2)^2$,
and some  work of the first author \cite{cat84}, \cite{singular} was focused on looking at the invariants and the deformations of bidouble covers,
deriving  basic results  for the moduli spaces of surfaces. 

Comessatti \cite{comessatti} was the first to study Galois coverings with abelian group $G$, and their relations to topology,
 while Pardini \cite{pardini}
described neatly the algebraic structure of such coverings, their invariants and deformations.

Just to give a flavour of the result: when $G$ is a cyclic group of order $n$ and $Y$ is factorial,  normal 
$G$-coverings correspond to 
 building data $(L, D_1, \dots , D_{n-1})$ consisting of reduced effective divisors 
 $D_1, \dots , D_{n-1}$
without common components, and of  the isomorphism class of a divisor $L$ such that $ n L \equiv \sum_i i D_i$
($\equiv$ is the  classical notation for linear equivalence).
The  theorem is the scheme counterpart of the field theoretic description 
$$  \CC(X) = \CC(Y) [z] / ( z^n - \Pi_i \de_i^{i}).$$
Here  $ D_i = \{ \de_i = 0\}$ and the branch locus $\sB_f$ is the union of the divisors $D_i$.

In Pardini's theorem  one can replace the field $\CC$ by any algebraically closed field of characteristic coprime to $n$, but over the complex numbers
we have  a more general result, namely, the extension of the Riemann existence theorem due to Grauert and Remmert:  normal schemes with a finite
covering $f : X \ra Y$, and with branch locus contained in a divisor $\sB$ correspond to conjugacy classes of monodromy 
homomorphisms $ \mu : \pi_1 (Y \setminus \sB) \ra \mathfrak S_n$. The scheme is a variety (i.e., irreducible) iff $\operatorname{Im} (\mu)$ is a transitive subgroup,
and Galois if it is simply transitive (similarly for Abelian coverings in \cite{BC08} the criterion for  irreducibility of $X$  was explicitly given).

On the other hand, the goal of finding explicit algebraic equations for the case of Galois coverings with non Abelian group $G$  is like the quest for the Holy Graal,
and in this paper, widely  extending  previous results of Tokunaga \cite{Tok}, we  give a full characterization of 
 Galois coverings  $\pi \colon X \ra Y$ with $Y$ smooth and with Galois 
group the dihedral group $D_n$ of order $2n$.

The underlying idea is of course the same and is easy to explain in general terms: we can factor $\pi \colon X \ra Y$ as the composition of  a cyclic covering 
of order $n$, $ p : X \ra Z : = X/H$,
(here $H \subset D_n$ is the group of rotations) and a (singular) double covering $q : Z \ra Y$.
A main new technical tool, developed in sections 3 and 4, consists in describing Weil divisors and the Picard group of divisorial sheaves
on a normal double cover $Z$. To make the results more appealing, we describe in detail the very special case of the Picard group
of hyperelliptic curves, especially the algebraic determination of torsion line bundles; this is an extension of a previous partial result by Mumford
{\cite{theta},  it boils down to determinantal equations for  triples of polynomials in one variable,  and bears interesting similarities with
the resultant of two polynomials.

In section 5 we describe the general theorem, whose  application in concrete cases is however  not so straightforward.  For this reason we concentrate in the next sections on two special 
very explicit classes of dihedral covers of algebraic varieties:
the simple and the almost simple dihedral covers. 

The underlying idea of simple covers is the one of giving a schematic equation which looks exactly as the equation describing   the field extension.
The first well known  instance is the one of a simple cyclic cover: $X$  is given by an equation
$$ z^n = F,$$ where $z$ is a fibre variable on the line bundle $\LL$ associated to a divisor $L$ on $Y$, and $ F $  is a section in  $H^0(\hol_Y ( n L))$.

In \cite{ENR} the analysis of everywhere nonreduced moduli spaces was based also on the technical notion of an almost simple cyclic cover,
given as the locus $X$  inside the $\PP^1$-bundle $\PP (\LL_0 \oplus \LL_{\infty})$, associated to a pair of two line bundles $\LL_0 , \LL_{\infty}$,
defined  by the equation
$$
z_0^n a_{\infty} =  z_{\infty}^n a_0, \ a_0 \in H^0 (\hol_Y ( n L_0)), \ a_{\infty} \in H^0 (\hol_Y ( n L_{\infty})),
$$
and where the divisors $E_0 = \{ a_0 = 0 \}$ and $E_{\infty} = \{ a_{\infty} = 0 \}$ are disjoint (else the covering is not finite).

Let us now explicitly describe the equations of a simple dihedral covering: $X$ is contained in a rank two split vector bundle of the form $\LL \oplus \LL$
over $Y$, and is defined by equations 
\begin{eqnarray*}
\begin{cases}
u^n + v^n &= 2a , \quad a \in   H^0 (\hol_Y ( n L)), \\
uv &= F  , \quad F \in   H^0 (\hol_Y ( 2 L)). 
\end{cases}
\end{eqnarray*}
The dihedral action is generated by an element $\s$ of order $n$ such that
$$  u \mapsto  \z u, v \mapsto \z^{-1} v, \ $$
where $\z$ is a primitive $n$-th root of $1$, and by an involution $\tau$ exchanging $ u $ with $v$.

The branch locus is here the divisor $\sB = \{ a^2 - F^n = 0\},$  $Z$ is defined by $z^2 = a^2 - F^n$, and the cyclic covering
$X$ is defined by $ u^n = z + a$: it is smooth if the two divisors $\{a=0\}$ and $\{ F = 0\}$ intersect transversally and $\sB$ is smooth outside of $ F=0$.

In particular, as it is well known from the work of Zariski \cite{Zar29}
and many other followers, the fundamental group $\pi_1 (Y \setminus \sB)$ is non Abelian and admits a surjection to the group $D_n$
(irreducibility is granted if there are points where  $\{a=0\}$ and $\{ F = 0\}$,   since there the covering is totally ramified).

In the case of simple dihedral coverings, the covering $ X \ra Z$ is only ramified in the $A_{n-1}$ singularities of $Z$, the points where $z=a=F = 0$.
In order to get a covering with more ramification we introduce then the almost simple dihedral covers, defined this time on the fibre product
of two $\PP^1$-bundles.
$X$ is the subset of  $ \PP ( \ul{\CC} \op \LL ) \times_Y \PP ( \ul{\CC} \op \LL  )$ 
defined by the following two equations:
\begin{eqnarray*}
\begin{cases}
u_1 v_1 - u_0 v_0 F & = 0 \\
a_\infty v_1^n u_0^n -2a_0 v_0^n u_0^n + a_\infty v_0^n u_1^n & = 0 \, ,
\end{cases}
\end{eqnarray*}
where, $v_1 , u_1$ are fibre coordinates on $\LL$,
$v_0, u_0$ are fibre coordinates on $\ul{\CC}$ (so that 
$([u_0 : u_1], [v_0 : v_1])$ are projective fibre coordinates on
$ \PP ( \ul{\CC} \op \LL ) \times_Y \PP ( \ul{\CC} \op \LL  )$ ).

As we said earlier, we want to find the invariants and study the deformations of the varieties that we construct in this way.
We first do this in full generality for any Galois covering in section 2; later we write more explicit formulae for simple and almost simple dihedral coverings.
We apply these in the case of easy examples, deferring the study of more complicated cases to the future.

One interesting application, to real forms of curves with cyclic symmetry, and related loci in  the moduli space of real curves,
  shall be given in a forthcoming  joint work with Michael L\"onne \cite{clp4}.

Finally, in this paper  we consider also the non Galois case: namely, coverings of degree $n$ whose Galois group is $D_n$.
In doing so, we establish connections
with the theory of triple coverings of algebraic varieties described by Miranda \cite{miranda} (if the triple covering is not Galois, then its Galois group is $D_3$!).


\section{Direct images of sheaves under Galois covers}
We consider (algebraic) varieties which are defined over the field of complex numbers $\CC$,
though many results remain valid for complex manifolds and for algebraic varieties defined over an algebraically 
closed field 
of characteristic $p$ not dividing $2n$.

\begin{defin}
Let $Y$ be a variety and let $G$ be a finite group. A {\bf Galois cover} of $Y$
with group $G$, shortly a {\bf $G$-cover} of  $Y$, is a finite morphism $\pi \colon X \ra Y$,
where $X$ is a variety 
 with an effective action by $G$, such that $\pi$ is $G$-invariant and induces an isomorphism
$X/G \cong Y$. 

A \textbf{dihedral cover} of $Y$ is a $D_n$-cover of $Y$, where $D_n$ is the dihedral group of order $2n$. 

If $\pi \colon X \ra Y$ is a $G$-cover, the \textbf{ramification locus} of $\pi$ is the locus $R: =R _\pi  \subset X$
of points with non-trivial stabilizer; the \textbf{(reduced) branch locus} of $\pi$ is 
$\sB_\pi:=\pi ( R ) \subset Y$.  

In the case where $X, Y$ are smooth, $R$ is the reduced divisor defined by the Jacobian determinant of $\pi$,
and also $\sB_\pi$ is a divisor (purity of the branch locus).
\end{defin} 

In the following, if there is no danger of confusion, we will denote $\sB_\pi$ by $\sB$.

Let $\pi \colon X \ra Y$ be a $G$-cover. 
Since $\pi \colon X \ra Y$ is finite, $\pi_*\Oh_X$ is a coherent sheaf of $\Oh_Y$-algebras.
The $\Oh_Y$-algebra structure on $\pi_*\Oh_X$ corresponds to a 
morphism $m\colon \pi_*\Oh_X \ot_{\Oh_Y} \pi_*\Oh_X \to \pi_*\Oh_X$
of $\Oh_Y$-modules that gives a commutative and associative product on $\pi_*\Oh_X$.
These data determine $X$ as $\spec \pi_* \Oh_X$. Furthermore, the $G$-action on $X$
corresponds to a $G$-action on $\pi_*\Oh_X$  which is the identity on $\hol_Y$.
The two structures on $\pi_*\Oh_X$, of an algebra over $\Oh_Y$ and of a $G$-module,
impose restrictions  that sometimes allow one to determine certain ``building data" for the $G$-cover $\pi$.
One of our aims here  is the determination of such building data when  $G=D_n$
(Theorem  \ref{STDn}).

We will restrict ourselves to the case where $\pi$ is flat, or equivalently  $\pi_*\Oh_X$
is locally free. We recall in the next proposition a useful characterization of
flat $G$-covers.

\begin{prop}\label{fcm}
Let $\pi \colon X \ra Y$ be a $G$-cover, with $X$ irreducible and 
$Y$ smooth. Then $\pi$ is flat $\Leftrightarrow$  $X$ is Cohen-Macaulay.
\end{prop}
\begin{proof}
($\Leftarrow$) This is a direct application of the Corollary to Theorem 23.1 in \cite{matsumura}.

($\Rightarrow$) Let $x\in X$ and $y = \pi ( x )$. Since $\pi$ is flat,
for any regular sequence $a_1, \ldots , a_r \in m_y$ for $\Oh_{Y,y}$,
$\pi^* (a_1), \ldots , \pi^* (a_r)$ is a regular sequence for $\Oh_{X,x}$.
Hence, since $Y$ is smooth and $\pi$ is finite, $X$ is  Cohen-Macaulay.
\end{proof}

In the rest  of this section we assume that  $X$  and $Y$ are smooth. 
The  aim here is to relate the basic sheaves of $X$, namely  $\Om_X^i , \Theta_X , \ldots$,
 with the corresponding ones of $Y$.  Concretely, for example for the  $\Om_X^i$'s,
there is a chain of inclusions as follows:
$$
\Om_Y^i \ot \pi_* \Oh_X \hookrightarrow \pi_* \Om_X^i \hookrightarrow  \Om_Y^i (\log \sB)^{**} \ot \pi_* \Oh_X \, ,
$$
where $\Om_Y^i (\log \sB)$ is the sheaf of $i$-forms with logarithmic poles along $\sB$
and $ \Om_Y^i (\log \sB)^{**}$ is its double dual. 
These three sheaves coincide on $Y\setminus \sB$ and it is possible to describe 
explicitly $ \pi_* \Om_X^i$ inside $\Om_Y^i (\log \sB)^{**} \ot \pi_* \Oh_X$ at the generic points $\sB_i$ of $\sB$.
Then, by Hartogs' theorem, one obtains a description of $ \pi_* \Om_X^i$.
We work out in detail the steps of this procedure for $\Om_X^1$ and $\Theta_X$.

Let $\CC ( X )$ and  $\CC ( Y )$ be the fields of rational 
functions  on $X$ and $Y$, respectively.  The space of differentials of $\CC ( X )$,
$\Om_{\CC ( X )}^1$, is a vector space of dimension 
$$
d=\dim_{\CC ( X )} \Om_{\CC ( X )}^1 = \dim X \, 
$$
and a basis  is given by the differentials of a transcendence basis of
the field extension  $\CC \subset \CC ( X )$.
The same holds true for $\Om_{\CC ( Y )}^1$ and, since $\dim X = \dim Y$,
the pull-back morphism $\pi^* \colon \CC ( Y ) \ra \CC ( X )$ induces an isomorphism
of $\CC ( X )$-vector spaces:
\begin{eqnarray}\label{star}
\Om_{\CC ( Y )}^1 \ot_{\CC ( Y )} \CC ( X ) \ra \Om_{\CC ( X )}^1 \, .
\end{eqnarray} 
In other words, the rational differentials of $X$ can be written as linear combinations 
$$
\sum_{i=1}^{d} f_i \operatorname{d}y_i \, , 
$$
where $f_1, \ldots , f_d \in \CC ( X )$,  $y_1 , \ldots , y_d \in \CC ( Y )$ form a   transcendence basis for
the extension $\CC \subset \CC ( Y )$. Since a regular differential 1-form on $X$ is a rational differential 
1-form which is regular at each point, the isomorphism \eqref{star} induces  an inclusion 
\begin{eqnarray}\label{starstar}
\varphi \colon \pi_*\Om^1_{X} \hookrightarrow \Om^1_{\CC ( Y )} \ot_{\CC ( Y )} \CC ( X ) \, .
\end{eqnarray}

Since $X$ and $Y$ are smooth, by the theorem of purity of the branch locus
\cite{Zar} the ramification locus $R$ and the branch locus $\sB$ of $\pi$
are divisors in $X$ and $Y$ respectively. 
Let $\sB=\sum_{i=1}^r \sB_i$, where $\sB_i \subset Y$ are prime divisors.
We recall, following \cite{CHKS06}, the following definition:
$$
\Om^1_Y (\log \sB) := {\rm Im} \left( \Om^1_Y \op \Oh_Y^{\op r} \to \Om_Y^1 (\sB) \right) \, ,
$$
where $ \Om^1_Y  \hookrightarrow \Om_Y^1 (\sB)$ is the natural inclusion, 
$\Oh_Y^{\op r} \to \Om_Y^1 (\sB)$ is given by
$$
e_i  \mapsto \frac{\operatorname{d}b_i}{b_i} \, , \quad i=1, \ldots , r \, ,
$$
where, for any $i$, $e_i$ is the $i$-th element of the standard 
$\Oh_Y$-basis of $\Oh_Y^{\op r}$, and $b_i$ is a global section of $\Oh_Y(\sB_i)$ with $\sB_i=\{b_i=0\}$.

Let $\Om_Y^1(\log \sB)^{**}$ be the double dual of $\Om_Y^1(\log \sB)$. It coincides with the 
sheaf of germs of logarithmic $1$-forms with poles along $\sB$ defined in \cite{saito}.
Indeed both sheaves are reflexive and they coincide on the complement $Y\setminus {\rm Sing}(\sB)$
of the singular locus of $\sB$.

\begin{prop}\label{pistaromega}
Let $\pi \colon X \ra Y$ be a $G$-cover with $X$ and $Y$ smooth.
Then there are $G$-equivariant inclusions of sheaves of $\Oh_Y$-modules as follows:
\begin{equation*}
\Om_Y^1\ot \pi_* \Oh_X \hookrightarrow  \pi_* \Om_X^1 \hookrightarrow \Om_Y^1(\log \sB)^{**} \ot \pi_* \Oh_X \, ,
\end{equation*}
where the morphisms  are isomorphisms on the complement  $Y\setminus \sB$
of the branch divisor.

Moreover, let $\xi_i \in \pi_* \Oh_X$ be the generator of the ideal of the reduced divisor $R_i : = \pi^{-1} (\sB_i)$ and a maximal local  root of $b_i$
($b_i = \xi_i ^{m_i}$), for any $i=1, \ldots , r$.
Then $ \pi_* \Om_X^1$  is characterized as the subsheaf of $\Om_Y^1(\log \sB)^{**} \ot \pi_* \Oh_X$
such that at the smooth points of $\sB_i =\{ b_i=0\}$ it coincides with the subsheaf of $\hol_Y$-modules generated by $  \Om_Y^1 \ot \pi_* \Oh_X$
and by the elements $\xi_i ^k \operatorname{d} \log (b_i), \ k = 1, \dots, m_i-1$.
\end{prop}

\begin{proof}
The first arrow to the left is the push-forward under $\pi_*$ of the  natural inclusion
$$
(T\pi)^*\colon \pi^*\Om_Y^1 \hookrightarrow \Om_X^1 \, ,
$$
where $T\pi$ is the tangent map of $\pi$.

Now consider the morphism $\varphi$ in   \eqref{starstar}.
We first show that $({\rm Im}\varphi)_q \subset (\Om_Y^1(\log \sB)^{**} \ot \pi_* \Oh_X)_q$
for any $q\in Y':= Y\setminus \sing (\sB)$.
On the complement of $\sB$,  $Y\setminus \sB$,  
this  follows from the projection formula,
since  $\pi$ is not ramified on $\pi^{-1}(Y\setminus \sB)$ and hence
$$
\pi_* \Om^1_X \cong \Om^1_Y \otimes \pi_* \Oh_X  \cong 
\Om_Y^1(\log \sB)^{**} \ot \pi_* \Oh_X \, \quad \mbox{on $Y\setminus \sB$} .
$$
Let now $q\in \sB' := \sB \setminus \sing (\sB)$ and $p\in \pi^{-1}(q)$.
Choose local coordinates $y_1 , \ldots , y_d$ at $q$ and $x_1 , \ldots , x_d$
at $p$, such that $\pi$ has the following local expression: 
$$
y_1 = x_1^m ,  y_2 = x_2 ,  \ldots ,  y_d = x_d  \, .
$$ 
Hence $\Oh_{X,p}= \Oh_{Y,q}[x_1]/(x_1^m - y_1)$, and
the stalk $(\Om_Y^1)_q$ is generated over $\Oh_{Y,q}$ by
$$
\operatorname{d}y_1,  \operatorname{d}y_2 ,  \ldots , \operatorname{d}y_d ,$$
while the stalk $(\Om_X^1)_p$ is generated over $\Oh_{X,p}$ by
$$
\operatorname{d}x_1 , \operatorname{d}x_2 , \ldots , \operatorname{d}x_d,
$$
with the obvious relations 
\begin{eqnarray*}
m  \cdot \operatorname{d} \log (x_1)  &=& \operatorname{d} \log (y_1 )\, , \quad \mbox{equivalently} 
\quad \operatorname{d}x_1= \frac{x_1}{my_1}\operatorname{d}y_1 \, , \\
\operatorname{d}x_i &=& \operatorname{d}y_i \, , \quad i\geq 2 \, . 
\end{eqnarray*}
From this it follows  that $(\pi_*\Om_X^1)_q$ is generated
as $\Oh_{Y,q}$-module by the elements of the $G$-orbit of
$$
\begin{cases}
x_1^i \operatorname{d}y_j  \, \, \quad   0\leq i \leq m-1 \, ,  2\leq j \leq d ,  \\
x_1^k\frac{\operatorname{d}y_1}{y_1} \, , \quad 1\leq k \leq m-1  \\ 
\operatorname{d}y_1\,
\end{cases}
$$  
which implies the claim at the points $q\in \sB'$,
since $y_1 = 0$ is a defining equation for $\sB$ near $q$.
Moreover, we see that $(\pi_*\Om_X^1)_q $ is generated,  as an $\hol_Y$-module, by $(\Om_Y^1)_q \ot (\pi_*\Oh_X)_q$
and by the elements $x_1^k\frac{\operatorname{d}y_1}{y_1} \, , \quad 1\leq k \leq m-1 $.

Notice that this argument also shows that the inclusion 
$(\pi_*\Om_X^1)_q \subset (\Om_{Y,q}(\log \sB)^{**}\ot (\pi_*\Oh_X))_q$ is $G$-equivariant, 
if $q\in Y\setminus \sing ( \sB )$. 

To conclude the proof, observe  that $\pi_* \Om^1_X$ is locally free and  that   
$\Om_Y^1(\log \sB)^{**} \ot \pi_* \Oh_X$
is reflexive, so that any section of  $\pi_*\Om_X^1$ on $Y\setminus {\rm Sing}(\sB)$
has a unique extension in $\Om_{Y}^1(\log \sB)^{**}\ot \pi_*\Oh_X$. 

\end{proof}

Let now $\Theta_X := \mathcal{H}om_{\Oh_X} (\Om_X^1, \Oh_X)$ and 
$\Theta_Y :=  \mathcal{H}om_{\Oh_Y} (\Om_Y^1, \Oh_Y)$ 
be the tangent sheaves of $X$ and $Y$, respectively. In the following proposition we
relate the sheaf $\pi_* \Theta_X$ with $\Theta_Y$.  

 Define as  in \cite{saito} or \cite[Def. 9.15]{cime88}    the sheaf $\Theta_Y(-\log \sB)$ 
 of logarithmic vector fields on $Y$ with respect 
to $\sB = \{ b=0\}$ as 
$$
\Theta_Y(-\log \sB) : = \{  v |  v \cdot  \log (b) \in \hol_Y\} =  \{  v |  v  (b) \in  b \hol_Y\} =  (  \Om_Y^1(\log \sB) )^*.
$$ 

Observe that the quotient sheaf $\Theta_Y  / \Theta_Y(-\log \sB)$ equals the equisingular normal sheaf 
$N'_{\sB|Y}$ of $\sB$ in $Y$, which coincides with the normal sheaf at the points where
$\sB$ is smooth (\cite{cime88}, rem. 9.16).

Let us recall that  the usual pairing $\Theta_Y \times \Om_Y^1 \to \Oh_Y$
extends to a  perfect pairing
$$
\Theta_Y(-\log \sB) \times \Om_Y^1 (\log \sB)^{**} \to \Oh_Y \, .
$$

\begin{prop}\label{pistartheta}
Let $\pi \colon X \ra Y$ be a $G$-cover with $X$ and $Y$ smooth.
Let $\sB\subset Y$ be the branch divisor of $\pi$.
Then the tangent map of $\pi$ identifies $\pi_* \Theta_X$ with a subsheaf of $\Theta_Y \ot \pi_* \Oh_X$.
Under this identification  we have the following inclusions of sheaves
$$
\Theta_Y(-\log \sB) \otimes \pi_* \Oh_X \subset \pi_* \Theta_X \subset  
\Theta_Y \otimes \pi_* \Oh_X \, 
$$
which are compatible with the $G$-actions.
Moreover the three sheaves coincide on $Y\setminus \sB$ and $\pi_* \Theta_X$ 
is characterized as the subsheaf of $\Theta_Y \otimes \pi_* \Oh_X$ sending $\pi_* \Omega^1_X$ to $\pi_* \Oh_X$.

More concretely, it is the subsheaf 
such that at the smooth points of $\sB_i =\{ b_i=0\}$  coincides with the subsheaf of $\hol_Y$-modules generated by $ \Theta_Y(-\log \sB) \otimes \pi_* \Oh_X $
and by $\xi_i ^{-1}  b_i \ \frac{\partial}{\partial b_i} $, where $\xi_i \in \pi_* \Oh_X$ is the generator of the ideal of the reduced divisor $R_i : = \pi^{-1} (\sB_i)$ (and a maximal local  root of $b_i$, $b_i = \xi_i ^{m_i}$),
while $b_i \ \frac{\partial}{\partial b_i} $ is a local generator of $N'_{\sB|Y}$.
\end{prop}

\begin{proof}
The tangent morphism $T\pi \colon \Theta_X \ra \pi^* \Theta_Y$ gives a $G$-equivariant morphism of sheaves
\begin{equation}\label{Tpi}
\pi_* \Theta_X \ra \Theta_Y \ot \pi_* \Oh_X \, .
\end{equation}
We first prove that \eqref{Tpi} is injective and that its image contains $\Theta_Y(-\log \sB) \otimes \pi_* \Oh_X$.
By Hartogs' theorem and the definition of $\Theta_Y(-\log \sB)$ it suffices to prove this 
on the  complement of $\sing ( \sB )$. 

 On $Y\setminus \sB$, $\pi$ is not ramified, hence \eqref{Tpi} is an isomorphism. 
Let now $q \in \sB\setminus \sing ( \sB )$ and let $p\in \pi^{-1}(q)$. Choose as before coordinates $y_1 , \ldots , y_d$
 at $q$ and $x_1 , \ldots , x_d$  at $p$, such that $\pi$ has the following local
expression:
$$
y_1 = x_1^m , y_2 = x_2 , \ldots , y_d = x_d \, .
$$
Then $(\pi_* \Theta_X)_q$ is generated by the $G$-orbit of 
$$
x_1^k\left( \frac{\partial}{\partial x_1} + mx_1^{m-1}\frac{\partial}{\partial y_1} \right) , 
x_1^k\frac{\partial}{\partial y_2} , \ldots , x_1^k\frac{\partial}{\partial y_d} , 0\leq k \leq m-1 ,
$$
 as $\Oh_{Y,q}$-module. The image of these generators under  \eqref{Tpi} is the $G$-orbit of
$$
m \frac{\partial}{\partial y_1} \ot x_1^{m-1+k}, \frac{\partial}{\partial y_2} \ot x_1^k , \ldots , 
\frac{\partial}{\partial y_d} \ot x_1^k , \, 0\leq k \leq m-1\, .
$$
From this it follows that \eqref{Tpi} is injective  and that its image contains the sheaf 
$\Theta_Y(-\log \sB) \otimes \pi_* \Oh_X$,
which is generated by the $G$-orbit of 
$$
y_1 \frac{\partial}{\partial y_1} \ot x_1^k = \frac{\partial}{\partial y_1} \ot x_1^{m+k}  , 
\frac{\partial}{\partial y_2} \ot x_1^k , 
\ldots , \frac{\partial}{\partial y_d} \ot x_1^k ,  \quad 0\leq k \leq m-1, 
$$
as $ \Oh_{Y,q}$-module. 
The only  missing term to get  from $\Theta_Y(-\log \sB) \otimes \pi_* \Oh_X$ the full direct image is then $y_1 \frac{\partial}{\partial y_1} \ot x_1^{-1}  $. We observe that 
$y_1 \frac{\partial}{\partial y_1} \ot x_1^{-1}  = \frac{1}{x_1} y_1 \frac{\partial}{\partial y_1}$ and 
the vector field $y_1 \frac{\partial}{\partial y_1}$  is the local generator of 
$N'_{\sB|Y} = \Theta_Y / \Theta_Y(-\log \sB)$.
\end{proof}

\begin{rem}
Notice that the previous results are valid for more general Galois covers of complex 
manifolds. 
\end{rem}

\subsection{The non-Galois case}
We  extend propositions \ref{pistaromega} and \ref{pistartheta}
to branched covers $\pi \colon X \ra Y$ which are not necessarily 
Galois. Notice that also in this case we have an inclusion of the sheaf 
$\pi_* \Om_X^1$ in the constant sheaf $\Om^1_{\CC ( Y )} \ot_{\CC ( Y )} \CC ( X )$
as in \eqref{starstar}.
\begin{prop}\label{pistaromegageneral}
Let $\pi \colon X \ra Y$ be a branched cover with $X$
and $Y$ smooth. Let $\sB\subset Y$ be the branch divisor 
of $\pi$. Then we have the following inclusions of $\Oh_Y$-modules:
$$
\Om_Y^1 \ot \pi_*\Oh_X \hookrightarrow 
\pi_* \Om_X^1 \hookrightarrow \Om_Y^1(\log \sB)^{**} \ot \pi_* \Oh_X \, ,
$$
where the morphism on the left is induced by the  pull-back morphism
$(T\pi)^* \colon \pi^* \Om_Y^1 \ra \Om_X^1$ and the one on the right is induced by  \eqref{starstar}.
\end{prop}
\begin{proof}
Since $\pi \colon X\setminus R \ra Y\setminus \sB$ is non-ramified,
$(T\pi)^* \colon \pi^*\Om_Y^1 \ra \Om_X^1$ is injective. Applying the push-forward 
functor $\pi_*$ we obtain an injective homomorphism $\Om_Y^1 \ot \pi_* \Oh_X 
\hookrightarrow \pi_*\Om^1_X$. 

We now prove  that the morphism \eqref{starstar} has image in $\Om^1_Y (\log \sB)^{**} \ot \pi_* \Oh_X$.
Notice that by the same argument as in the proof of Proposition \ref{pistaromega}, it suffices to prove
this on the complement of the singular locus of $\sB$.
So  we assume that $\sB$ is smooth. The assertion is true on $Y\setminus \sB$, because 
here $\pi$ is unramified. Let now $q\in \sB$ and let $V\subset Y$ be an open neighbourhood 
(in the complex topology) of $q$, such that: 
$$
\pi^{-1}(V) = U_1 \sqcup \ldots \sqcup U_s \, ,
$$ 
with $U_1, \ldots , U_s \subset X$  disjoint subsets; for any $i=1, \ldots , s$  there are coordinates
$(x_{i,1}, \ldots , x_{i,d})$ on $U_i$ and $(y_1 , \ldots , y_d)$ on $V$, such
that the restriction of $\pi$ on $U_i$ has the form
$$
(x_{i,1}, \ldots , x_{i,d}) \mapsto (y_1 = x_{i,1}^{e_i}, y_2 = x_{i,2},  \ldots , y_d=x_{i,d}) \, , 
$$
$e_1, \ldots , e_s$ are integers $\geq 1$ (see e.g. \cite{Namba}).
Then $(\pi_* \Om_X^1)_{|V}$ is generated as $\Oh_Y ( V )$-module by:
$$
\begin{cases}
x_{i,1}^{k_i}\operatorname{d}x_{i,1} = \frac{x_{i,1}^{k_i+1}}{e_i}\frac{\operatorname{d}y_1}{y_1} \, , \quad k_i = 0, \ldots , e_i -1
\, , i=1, \ldots , s \\
x_{i,1}^{k_i}\operatorname{d}x_{i,j} = x_{i,1}^{k_i}\operatorname{d}y_j \, , \quad j=2 , \ldots , d \, , i = 1, \ldots , s \, , k_i = 0, \ldots , e_1 
-1\, . 
\end{cases}
$$
Since $y_1 = 0$ is a local equation for $\sB$, it follows that the image of $\pi_* \Om^1_X$
under \eqref{starstar} is contained in $\Om_Y^1 (\log \sB)^{**} \ot \pi_* \Oh_X$.
\end{proof}
Analogously, for the tangent sheaves we have the following 
\begin{prop}
Let $\pi \colon X \ra Y$ be a branched cover, with $X$ and $Y$
smooth. Let $\sB\subset Y$ be the branch locus. Then
the tangent map of $\pi$ induces an injective morphism of sheaves
$$
\pi_* \Theta_X \hookrightarrow \Theta_Y \ot \pi_* \Oh_X
$$ 
whose image contains the sheaf $\Theta_Y (-\log \sB) \ot \pi_* \Oh_X$.
\end{prop} 

\begin{rem}We can obtain  more precise results by taking the Galois closure $p \colon W \ra Y$, which is a normal variety,
and writing $X = W / H$ where $H$ is a suitable subgroup of the Galois group $G$.
The only problem is that $W$ may not be smooth, and we denote then by $W'$ its smooth locus.
Then we can write $\pi_* \Om_X^1$ as essentially the submodule of $H$-invariants inside $p_* \Om_W^1$: for instance
we have
$$ \pi_* \Om_X^1 =  ((p_* \Om_W^1)^H )^{**},$$
since the regular 1-forms on $X' = W' /H$ are just the $H$-invariant 1-forms on $W'$.

\end{rem}

\section{Line bundles and divisorial sheaves on flat double covers}

In this section we  consider the following  general situation. We have a  flat finite double cover, where $Y$ is smooth: 

$$ q : Z \ra Y, \ Z = \operatorname{Spec} (q_* \hol_Z) ,$$
$$  \sR := q_* \hol_Z   = (\hol_Y \oplus z \hol_Y (-L)/ (z^2 - F)), \ F  \in H^0 (\hol_Y (2L)).$$ 

Our goal is to have  a description of divisorial sheaves, that is, rank 1  reflexive sheaves $\sL$ on the Gorenstein variety $Z$
in terms of their direct image $q_* ( \sL) = : \sN.$

By flatness of $\sL$ over $Y$, which we assume throughout this section, 
$\sN$ is a rank two vector bundle on $Y$ (a locally free sheaf of rank 2),
and its $\sR$- module structure is fully encoded in an endomorphism 
$$ N : \sN (-L) \ra \sN $$
such that  $$ N^2  = F \cdot  \operatorname{Id}_{\sN} : \sN (-2L) \ra \sN .$$
In particular, observe that $\operatorname{Tr} (N) = 0, \det (N) = - F$,
so that, on any open set where $\sN$ and the divisor $L$ are   trivialized,
$N$ is given by a matrix  of the form
$$
\left(
\begin{matrix}
a & b \\
& \\
c & - a
\end{matrix}
\right) \ ,  
\ a^2 + bc =  F.
$$
Conversely, any such pair $(\sN, N)$ as above determines a divisorial  sheaf $\sL$ on $Z$.

\begin{lemma}\label{free}
Given a pair  $(\sN, N)$ where $ N : \sN (-L) \ra \sN $ satisfies $ N^2  = F \cdot  \operatorname{Id}_{\sN} : \sN (-2L) \ra \sN ,$
it determines a saturated rank 1 torsion free $\sR$-module which is locally free exactly  in the points of $Y$
where the endomorphism $N$ does not vanish. In particular, a divisorial sheaf $\sL$ on $Z$.
\end{lemma}
\proof
View locally $\sN$ as $\hol_Y^2$: then a local section $(x,y)$ is a local $\sR$ generator at a point $P$ if and only if 
$ (x,y) $ and $N(x,y)$ are a local basis. This means, in terms of the matrix $N$,
that $ (x,y) $  and $(ax + b y, cx - ay)$ are linearly independent; equivalently,
the determinant 
$$ q(x,y) =   c x^2 - 2 a xy - b y^2 \neq 0.$$ 
Hence $\sN$ is an invertible $\sR$-module if and only if the  quadratic form obtained by evaluating $a,b,c$ at $P$ is not identically zero,
which amounts to the condition that $a,b,c$ do not vanish simultaneously.

\qed
\begin{rem}
Of course, the points where $N$ vanishes yield singular points of the branch locus $\sB = \{ F = 0\}$,
and singular points of $Z$.
\end{rem}

In the case where $Z$ is normal, then $\sL$ is determined by its restriction on the smooth locus $Z^0$ of $Z$
since $\sL = i_* (\sL|_{Z^0})$ (we are denoting by $ i : Z^0 \ra Z$ the inclusion map),
 hence we are just dealing with the Picard group $\operatorname{Pic}(Z^0)$. We want now 
to spell out in detail the  group structure of $\operatorname{Pic}(Z^0)$ in terms of the direct image rank 2 vector bundles.

\begin{prop}\label{picz}
Let $Z$ be normal,  let $\sL_1^0, \sL_2^0$ be invertible sheaves on $Z^0$, and let 
$\sN_j = (q \circ i)_* (\sL_j^0), \ j = 1,2.$ 
Then   the tensor product operation  in $\operatorname{Pic}(Z^0)$,
associating to $\sL_1^0, \sL_2^0$ their tensor product $\sL_1^0 \otimes_{\hol_{Z^0}} \sL_2^0$
gives the following pair:

1) the vector bundle $(q \circ i)_* (\sL_1^0 \otimes_{\hol_{Z^0}}  \sL_2^0 )$ equals the $\hol_Y$-double dual of $ \sN_1 \otimes_{\sR}\sN_2  $,
which is the  cokernel of the following exact sequence:
$$   \psi : (\sN_1 \otimes_{\hol_Y}\sN_2 )(-L) \ra  \sN_1 \otimes_{\hol_Y}\sN_2 \ra \sN_1 \otimes_{\sR}\sN_2 \ra 0,$$
where $$\psi = N_1  \otimes_{\hol_Y}  \operatorname{Id}_{\sN_2 } -  \operatorname{Id}_{\sN_1} \otimes_{\hol_Y} N_2 .$$ 

2) To the inverse invertible sheaf ${\sL^0}^{-1}$ corresponds the $\sR$-module associated to the pair
$(\sN^* (-L), ^tN(-2L))$, where $\sN^* : = \sH om (\sN, \hol_Y)$.

3) Two invertible sheaves on $Z^0$, $\sL_1^0, \sL_2^0$, are isomorphic if and only if there is an isomorphism
$\Psi : \sN_1 \ra \sN_2$ such that $\Psi \circ N_1 = N_2 \circ \Psi (-L)$.

4) Effective Weil divisors $D$ on $Z$ correspond to points $[\de]$ of some projective space  $\PP (H^0(\sN))$, for some isomorphism class of 
a pair $(\sN, N)$ as above. The sum $D_1 + D_2$ corresponds to the image of $\de_1 \otimes_{\sR} \de_2$ in the double dual
$(\sN_1 \otimes_{\sR}  \sN_2)^{**}$.
\end{prop}

\proof
1) First of all  we have that $(q \circ i)_* (\sL_1^0 \otimes_{\hol_{Z^0}} \sL_2^0 )$ is the saturation of the sheaf equal to $\sN_1 \otimes_{\sR}\sN_2  $
on the open set $Y^0 = Y \setminus \operatorname{Sing} (\sB)$ (letting $j : Y^0 \ra Y$ be the inclusion, the saturation of $\sF$ is here for us $j_* (\sF|_{Y^0}$)).

Moreover, by   definition, $\sN_1 \otimes_{\sR}\sN_2  $ is the quotient of $ \sN_1 \otimes_{\hol_Y}\sN_2  $
by the submodule generated by the elements $ (z \cdot n_1)  \otimes_{\hol_Y} n_2 -  n_1 \otimes_{\hol_Y} ( z \cdot n_2) $,
i.e., the submodule  image of 
$$N_1  \otimes_{\hol_Y}  \operatorname{Id}_{\sN_2 } -  \operatorname{Id}_{\sN_1} \otimes_{\hol_Y} N_2 :   
\sN_1 \otimes_{\hol_Y}\sN_2 (-L) \ra  \sN_1 \otimes_{\hol_Y}\sN_2. $$ 
Since the above is  an antisymmetric map of vector bundles, its rank at one point is either $2$, or $0$: but the latter  happens, as it is easy to verify,
 exactly when both $N_1$ and $N_2$ vanish, 
i.e., in a locus contained in the singular locus of the branch locus $\sB  = \{ F=0\}$, which has codimension 2 since we assume $Z$ to be normal.

Hence we have that $ \sN_1 \otimes_{\sR}\sN_2  $ is a rank 2 bundle at $y \in Y$  if either $N_1$ or  $N_2$ do not vanish at $y$; in the contrary case, 
the sheaf $ \sN_1 \otimes_{\sR}\sN_2  $ has rank 4 at the point $y$, and one needs to take  $( \sN_1 \otimes_{\sR}\sN_2 )^{**} $,
where we set $\sF^* : = \sH om_{\hol_Y} (\sF , \hol_Y)$.

We observe moreover that we have an infinite complex

$$ \dots  \ra (\sN_1 \otimes_{\hol_Y}\sN_2 )(-2L) \ra   (\sN_1 \otimes_{\hol_Y}\sN_2 )(-L) \ra  \sN_1 \otimes_{\hol_Y}\sN_2 \ra \sN_1 \otimes_{\sR}\sN_2 \ra 0,$$
where the  maps are given by 
 $$\psi' = N_1  \otimes_{\hol_Y}  \operatorname{Id}_{\sN_2 } +  \operatorname{Id}_{\sN_1} \otimes_{\hol_Y} N_2 ,$$ 
respectively  by the above $\psi$.
The complex is exact on the right and also at the points where either $N_1$ or  $N_2$ are not vanishing.

2)  Given the module $ \sH om_{\sR} (\sN, \sR)$, its elements consist of elements  $\phi = \phi_1 + z  \phi_2$,
where $\phi_1 \in \sN^*$, and $\phi_2 : \sN \ra  \hol_Y(-L)$.
$\sR$-linearity is equivalent to $\phi (N n ) = z \phi (n)$, hence to
$$\phi_2 (N n) =  \phi_1(n) , \phi_1 (Nn) = F \phi_2(n).  $$
The first equation  implies the second, hence $\phi_2$ determines $\phi$; moreover, given any 
$\phi_2 : \sN \ra   \hol_Y(-L)$, and defining $  \phi_1(n) : = \phi_2 (N n) $, we get an $\sR$-linear homomorphism by the above formulae.

Clearly, multiplication by $z$ acts  on $\phi$ by sending $\phi = \phi_1 + z  \phi_2 \mapsto F \phi_2 + z \phi_1$,
hence it sends 
$\phi_2 :  \sN \ra   \hol_Y(-L)$ to $\phi_2 \circ N$, and it is given by $^t N (-2L) : \sN^* (-L) (-L) \ra \sN^* (-L)$.

3) Is easy to show, applying $(q \circ i)_*$ to the given isomorphism.

4) A section of $\sL$ determines, again by applying $(q \circ i)_*$, a section $\de \in H^0(\sN)$, which represents the image of $1 \in \hol_Y$
and which  in turn determines the image of an element $\alpha + z \beta \in \sR$: 
since we must  have $\alpha + z \beta \mapsto \de (\alpha) + N \de (\beta)$.

\qed

\section{Line bundles on hyperelliptic curves}

We apply the theory developed in the previous section to discuss the particular, but very interesting case, of line bundles on hyperelliptic curves.
Here $Y = \PP^1$, and $Z = C$ is a hyperelliptic curve of genus $g \geq 1$ defined by the equation 
$$  z^2 = F(x_0,x_1),$$
where $F$ is a homogeneous polynomial of degree $ 2g + 2$ without multiple roots.

The results of this section are also very relevant in order to show the complexity of the equations defining  
dihedral covers. In fact,  a $D_n$-covering of $\PP^1$, $ X \ra \PP^1$,  factors through a hyperelliptic curve $C$.
In the case where $X \ra C$ is \'etale, $X$  is determined by the choice of a line bundle $ L \in \operatorname{Pic}^0 (C)$ which is an element of $n$-torsion.
This is the reason why we dedicate special attention to the description of $n$-torsion line bundles on hyperelliptic curves.

For shorthand notation we write $\hol (d) : = \hol_{\PP^1} (d)$
and we observe that every locally free sheaf on $\PP^1$ is a direct sum of some invertible sheaves $\hol (d_j)$. 

Our {first} goal here is to parametrize the Picard group of $C$.
The first preliminary remark is that it suffices to parametrize the subsets $\operatorname{Pic}^0(C), \operatorname{Pic}^{-1} ( C )$, 
since for any line bundle $\sL$ we can choose an integer $d$ such that
$$ \deg (\sL \otimes_{\hol_C} q^* (\hol(d)) \in \{0, -1\}.$$

Let us then consider  a line bundle $\sL $ of degree zero or $-1$.
In the  case $\sL = \hol_C$ we know that the corresponding pair is the
vector bundle 
$$ \hol \oplus  z \hol (-g-1),  N (\alpha + z \beta) = F \beta + z \alpha,$$
and the corresponding matrix is
$$ N = 
\left(
\begin{matrix}
0 & F \\
& \\
1 & 0
\end{matrix}
\right) \
$$
whose determinant yields a trivial factorization  $F = 1 \cdot F$.

If instead we assume that $\sL$ is nontrivial, i.e. it has no sections, we can write
$$\sN : = q_* (\sL) = \hol (-a) \oplus \hol (-b), \   0 < a \leq b,  \ a + b = g+1 -d, \ d = \deg (\sL).$$

The integer $a$ is  equal to
$$ \min \{ m | H^0 (\sL \otimes q^* \Oh (m)) \neq 0 \},$$ 
and it determines a Zariski locally closed stratification of the Picard group.  We fix $a$ and $ d \in \{0,1\}$ for the time being,
and observe that then the vector bundle $\sN$ is uniquely determined, whereas the matrix $N$ (determining the $\sR$-structure)
has the form
$$ N = 
\left(
\begin{matrix}
P & f \\
& \\
q & -P
\end{matrix}
\right) \ , 
- \det (N)  =  (P^2 + qf) = F .
$$
Here,   $$ \deg (P) = g+1, \deg (f) = g + 1 - a + b, \deg (q) = g + 1 + a - b .$$
Recall in fact  that, in terms of fibre variables $(u,v)$, we have $$ z u = P u + q v , z v = f u - P v.$$

We see the matrix $N$ as providing a factorization 
$$ F - P^2  = q f. $$ 

It follows that $N$ is fully determined by the degree $g+1$ polynomial $P$ and by a partial factorization
of the polynomial $ F - P^2 =   \Pi_1^{2g+2} l_i (x_0, x_1)$ (here the $l_i$ are linear forms),
as the product $  f q$ of two polynomials of respective degrees $\deg (f) = g + 1 - a + b \geq g+1, \deg (q) = g + 1 + a - b \leq g+1$.
Observe that  the choice for the polynomial $q$ (hence of $f$)  is only unique up to a constant $\la \in \CC^*$, once the partial  factorization is fixed.

Since for  $P$ general the linear forms $l_i$ are distinct, we have in this case exactly
$ \frac{(2g+2)!}{(g + 1 - a + b)! (g + 1 + a - b)!}$ such possible factorizations.

Denote by $\sV(a,b)$ the variety of such matrices:  $\sV(a,b) = \{(P,f,q) | P^2 + qf  = F  \}$. The preceding discussion shows that the variety
$\sV(a,b)$ parametrizing such 
 matrices $N$ is  a $\CC^*$ bundle over a finite covering of 
degree $ \frac{(2g+2)!}{(g + 1 - a + b)! (g + 1 + a - b)!}$ of the affine space $\CC^{g+2}$ parametrizing  our polynomials $P$.
 Hence $\sV(a,b)$ is an affine variety  of dimension $g+3$. However, in order to get isomorphism classes of line bundles (elements of the Picard group)
 we must divide by the adjoint action of the group $\sG$ of the automorphisms of the vector bundle 
 $\sN : = q_* (\sL) = \hol (-a) \oplus \hol (-b)$ which have determinant $1$.
 
 For $a=b$ we get $\sG = SL (2, \CC)$, whereas for $a < b$ we get a triangular group of matrices
 $$ B = 
\left(
\begin{matrix}
\la & \beta \\
& \\
0 & \la^{-1}
\end{matrix}
\right) \ , 
$$
where $\beta$ is any homogeneous polynomial of degree $b-a$.

We also observe that the stabilizer of a matrix $N$ corresponds to an isomorphism of $\sL$,
hence to a scalar $\mu \in \CC^*$ whose square equals $\det (B) = 1$, hence $\mu = \pm 1$.

We do not investigate here the GIT  stability of the orbits, but just observe that the
$\sV(a,b)$ yields an open set of dimension $g$ in the case where $ a=b$,
else it gives a stratum of dimension $ g + 1 - (b-a)$.

Recalling  that  $a + b = g+1 -d $, for  $d=0$ the open set corresponds to:  the case $a=b= \frac{g+1}{2}$ 
for $g$ odd, and the case $a =  \frac{g}{2}, b =  \frac{g}{2} + 1$ for $g$ even; 
similarly for $ d = -1$.

Next we investigate the explicit description of tensor powers of an invertible sheaf $\sL$ on the hyperelliptic curve $C$,
with two motivations: the first one  is to try to get useful  results  towards the description of dihedral coverings of the projective line,
the second in order to describe torsion line bundles on hyperelliptic curves.

\begin{prop}\label{power}
There are exact sequences
$$ 0 \ra \sK_2 : = \hol (-a - b -(g+1)) \ra S^2 (q_* \sL) \ra q_* (\sL ^{2} ) \ra 0,$$
$$ 0 \ra \sK_3  := \sK_2 \otimes q_* \sL   \ra S^3 (q_* \sL) \ra q_* (\sL ^{3} ) \ra 0, $$ 
$$ 0 \ra \sK_n  := \sK_2 \otimes S^{n-2} (q_* \sL)    \ra S^n (q_* \sL) \ra q_* (\sL ^{n} ) \ra 0. $$ 
\end{prop} 
\proof
The local sections of $\sN = q_* \sL$ can be written as pairs $(u,v)$.
Hence $q_* (\sL ^{2} )$ is generated by $ u^2, uv, v^2$, subject to the relation
$$u (fu - P v) =  u (z v) = z (uv) = v (z u) = v (Pu + q v) \Leftrightarrow $$
$$ \Leftrightarrow  \Xi : = f u^2 - 2 P uv - q v^2 = 0.$$ 
That this is the only relation follows since the kernel $\sK_2$ has rank $1$ and first  Chern class equal to $-a-b-(g+1)$.

Similarly, $q_* (\sL ^{3} )$ is generated by the cubic monomials $u^3, u^2 v , u v^2, v^3$, 
and we can simply multiply the relation $\Xi$ by $u$, respectively $v$.

In general, we simply observe that there is a natural morphism 
$$ q_{\sL} : C \ra \operatorname{Proj} (q_* \sL)$$ given by evaluation, and whose image is 
the relative quadric  $\Ga : = \{ (u,v) | \Xi (u,v) = 0\}$.

Recall that, since $\sL$ is invertible, the polynomials
$f,P,q$ cannot vanish simultaneously, in particular $\Ga : = \{ \Xi (u,v) = 0\}$ is irreducible.
Moreover the branch locus of $\Ga \ra \PP^1$ is $\{ (x_0, x_1) | fq + P^2 = F= 0\}$,
therefore $\Ga$ 
is  isomorphic to $C$.

 Hence $C$ is the hypersurface in $\PP' : = \operatorname{Proj} (q_* \sL)$ defined by $\Xi = 0$,
 where $\Xi$  is a section of $ \hol_{\PP'} (2) \otimes q^* (\sK_2)^{-1}$;
 hence we obtain the general exact sequence via the pushforward of
 $$ 0 \ra  \hol_{\PP'} (n) ( -C) =  \hol_{\PP'} (n-2) \otimes q^* (\sK_2)  \ra \hol_{\PP'} (n) \ra \hol_C (n) \ra 0 .$$ 

\qed

\begin{cor}
A  line bundle $\sL$ of $n$-torsion on the hyperelliptic curve $C$ corresponds to a pair 
$ (\sN , N)$, $$\sN= \hol (-a) \oplus \hol (-b), \   0 < a \leq b,  \ a + b = g+1, $$
$$ N = 
\left(
\begin{matrix}
P & f \\
& \\
q & -P
\end{matrix}
\right) \ , 
\det (N)  = - (P^2 + qf) = - F .
$$
where $\deg (P) = g+1, \deg (f) = 2 b, \deg (q) = 2 a  ,$
such that the following  linear map 
$$     H^0 (S^n (\hol(a) \oplus  \hol(b)) (-2))  \ra   H^0 (S^{n-2} (\hol(a) \oplus  \hol(b))(2g)) ,$$
equal to the dual of $\sK_2 \otimes S^{n-2} (q_* \sL)    \ra S^n (q_* \sL) $ twisted by $(-2)$, is not surjective.
\end{cor}

\proof
Let $\sL$ be a degree zero line bundle on $C$, so that we have $ a + b= g+1$.
The condition that $ \sL ^{n} $ is trivial is clearly equivalent to the condition $ H^0 (q_* (\sL ^{n} )) \neq 0.$
In view of the exact cohomology sequence  (here $ H^0 (S^n (q_* \sL) ) = H^0 (S^n ( \sN) )  = 0$ since $a, b > 0$):
$$ 0 \ra H^0 (\sK_n   ) = 0    \ra H^0 (S^n (q_* \sL) ) = 0 \ra H^0 (q_* (\sL ^{n} ))   \ra  $$
$$ \ra H^1 (\sK_n   )     \ra H^1 (S^n (q_* \sL) )  \ra H^1 (q_* (\sL ^{n} ))   \ra  0, $$ 
the condition $ H^0 (q_* (\sL ^{n} )) \neq 0$ is equivalent to the non injectivity of $H^1 (\sK_n  )     \ra H^1 (S^n (q_* \sL) )$,
equivalently, to the non surjectivity of the homomorphism of Serre dual vector spaces:
$$     \coker (H^0 (S^n ( \sN^* )(-2) )   \ra     H^0 ( S^{n-2} ( \sN^*)  (2g)) ) \neq 0 \Leftrightarrow$$
$$    \Leftrightarrow  \coker ( H^0 (S^n (\hol(a) \oplus  \hol(b)) (-2))  \ra   H^0 (S^{n-2} (\hol(a) \oplus  \hol(b))(2g)) ) \neq 0.$$
 
For $n=2$ this means that, denoting  by $\sA[m] = H^0 (\hol(m))$ the space of homogeneous polynomials of degree $m$, the linear map 
$$ \sA[2a-2] \oplus \sA[a+b -2] \oplus \sA[2b-2] \ra \sA [ 2a + 2b-2]$$
given by the matrix $( f, -2 P, q)$ is not surjective (equivalently, the linear map given by the matrix $( f,  P, q)$ is not surjective).

Writing in non homogeneous coordinates $$ f = \sum_i^{2b} f_i x^i, P = \sum_i^{a+b} P_i x^i, q= \sum_i^{2a} q_i x^i,$$
this means that the matrix
$$ A (f,P,q) = 
\left(
\begin{matrix}
f_0 & f_1 & \dots & f_{2b-1}& f_{2b}& 0 & 0 & \dots & 0\\
& \\
0 &  f_0 & f_1 & \dots & f_{2b-1}& f_{2b}& 0 & \dots & 0\\
& \\
\dots &  \dots & \dots & \dots & \dots & \dots & \dots & \dots & \dots\\
& \\
0 & 0 &0 & 0 &  f_0 & f_1 & \dots & f_{2b-1}& f_{2b}\\
& \\
P_0 & P_1 & \dots & P_{a+b }&0 & 0 & 0 & \dots & 0\\
& \\
0 &  P_0 & P_1 & \dots & P_{a+b }&0 &  0 & \dots & 0\\
& \\
\dots &  \dots & \dots & \dots & \dots & \dots & \dots & \dots & \dots\\
& \\
0 & 0 &0 & 0 & 0 &  P_0 & P_1 & \dots & P_{a+b }\\
& \\
q_0 & q_1 & \dots & q_{2a}& 0 & 0 & 0 & \dots & 0\\
& \\
0 &  q_0 & q_1 & \dots & q_{2a}&  0 & 0 & \dots & 0\\
& \\
\dots &  \dots & \dots & \dots & \dots & \dots & \dots & \dots & \dots\\
& \\
0 & 0 &0 & 0 & 0 &  q_0 & q_1 & \dots & q_{2a }\\

\end{matrix}
\right) \ , 
$$
does not have maximal rank.

\qed

\begin{rem}
The reader may notice the similarity of the matrix $A(f,P,q)$ with the matrix giving the resultant of two homogenous polynomials
in two variables. Moreover, notice that $A(f,P,q)$ is a $ (3a + 3b -3) \times ( 2a + 2b -1)$ matrix, hence the condition that
its rank be at most $ 2a+ 2b -2$ amounts to a codimension $ g = a+b -1$ condition, which is the expected codimension
of the set of  n-torsion points in $\operatorname{Pic}^0 (C)$.
\end{rem}

\subsection{Powers of divisorial  sheaves on double covers}
We would now like to show how the result of proposition \ref{power} generalizes to any flat double cover $ q : Z \ra Y$.
Let $\sL$ be a divisorial  sheaf on $Z$, and $(\sN , N)$ the associated pair, where $\sN = q_* (\sL)$.
There is a natural map
$ q_{\sL} : Z \dasharrow \operatorname{Proj}(\sN)$ given by evaluation, and which is a morphism on the smooth locus of $Z$.

Set $\PP' : = \operatorname{Proj}(\sN)$, and let $\Ga$ be the image of $ q_{\sL} $. Clearly $Z$ is birational to $\Ga$, which is a Gorenstein variety,
since it is a divisor in $\PP'$, and $ q_{\sL} $ is bijective outside of the inverse image of the branch locus $\sB_q$.
Moreover $q$ factors through $ q_{\sL} $ and the natural projection $\pi : \PP' \ra Y$. 

At the points where $\sL$ is invertible, then $\sL $ is isomorphic with $\hol_Z$, with an isomorphism compatible with the projection $q$, hence
we conclude that $ q_{\sL} $ is an embedding on the smooth locus of $Z$.

At a point $P$ where $\sL$ is not invertible, we use the local description of the pair $(\sN , N)$ as $(\hol_Y^2, N)$,
where $N$ is the matrix sending two local generators $x, y$  to $a x + c y ,$ respectively to $bx -a y$.
Since $\sL$ is not invertible, by lemma \ref{free} we get that $a, b,c$ vanish at $P$.
We use now the relations
$$ x z = ax + cy \Leftrightarrow x (z-a) = cy , \ y z  = b x - a y \Leftrightarrow y (z+a) = b x$$
to obtain 
$$ ( x : y ) = ( c : z-a) = (z+a : b) .$$  

These formulae (observe that $( c : z-a) = (z+a : b) $ since $z^2 - a^2 = bc$) show that, at the points $P$ where $a,b,c$ (hence also $z$) vanish, the rational map $ q_{\sL} $ blows up the point $q^{-1}(P)$
to the whole fibre $\PP^1$ lying over $P$.
 Hence $\Ga$ is a small resolution of $Z$,
and we have that the inverse of $ q_{\sL} $ is obtained blowing down these $\PP^1$' s  lying over such points $P$.

Since we have an isomorphism $Z^0 \cong \Ga^0$, we see that the line bundle $\hol_{\Ga}(1)$ restricts to $\sL$ on $Z^0$,
hence the sections of $\sL^n$ on $Z^0$ correspond to sections of  $\hol_{\Ga}(n)$ on $ \Ga^0$.

 Notice that $\operatorname{Pic} (\PP')$ is generated by $\operatorname{Pic}(Y)$ and by $ \hol_{\PP'} (1)$,
hence there is an invertible sheaf $ \sK_2 $ on $Y$ such that $\Ga$ is the zero set of  a section of $ \hol_{\PP'} (1) \otimes \pi ^* (\sK_2)^{-1}$.

Consider now the exact sequence
 $$ 0 \ra  \hol_{\PP'} (n) ( - \Ga) =  \hol_{\PP'} (n-2) \otimes \pi ^* (\sK_2)  \ra \hol_{\PP'} (n) \ra \hol_{\Ga} (n) \ra 0 .$$ 
 Taking the direct image, and observing that $\pi_* \hol_{\Ga} (n) = q_* (\sL^n)$, we obtain the following
 
 \begin{prop}\label{kn}
 For each divisorial sheaf $\sL$ on $Z$
 there is an exact sequence
 $$ 0 \ra \sK_n  := \sK_2 \otimes S^{n-2} (q_* \sL)    \ra S^n (q_* \sL) \ra q_* (\sL ^{n} ) \ra 0. $$ 
 \end{prop}

\section{Dihedral field extensions and generalities on dihedral covers}
In this section we describe dihedral field extensions 
$\CC( Y ) \subset \CC ( X )$   in the case where 
$Y$ is factorial and $X$ is normal, thus we obtain 
a birational classification of  $D_n$-covers. 
Recall that for any 
$G$-cover $\pi \colon X \ra Y$ of normal varieties 
the field extension $\CC ( Y ) \subset \CC ( X )$ is Galois 
with group $G$ (a $G$-extension); on the other hand
any such  field extension  determines $\pi$
as the normalization of $Y$ in $\CC ( X )$.

In the second part of this section we 
determine the geometric building data that make the
above normalization process explicit.
This is important, for instance  to calculate 
invariants of $X$, to determine the direct images of basic sheaves
on $X$ (e.g. $\pi_* \Oh_X, \pi_* \Om_X, \pi_*\Theta_X$, etc.)
and to provide explicitly families of such covers.

\subsection{Dihedral field extensions}\label{Dfe}

Let us fix the following presentation for the dihedral groups:
$$
D_n = \langle \s , \tau \, | \, \s^n = \tau^2 = (\s\tau)^2=1 \rangle \, .
$$
The   split exact sequence
$$
0 \ra  \langle \s \rangle = \ZZ/n\ZZ \ra D_n \ra \ZZ/2\ZZ \ra 0 \, 
$$
gives a factorization of the  $D_n$-cover $\pi \colon X \ra Y$  
as the composition of two intermediate cyclic covers: $\pi = q\circ p$, where
$p\colon X \ra Z$ is a $\ZZ/n\ZZ$-cover,
$q\colon Z \ra Y$ is a $\ZZ/2\ZZ$-cover, $Z:= X/\langle \s \rangle$.
If $X$ is normal and irreducible, then also $Z$ and $Y$ are so, and
we have the following chain of field extensions:
$$
\CC (Y ) \subset \CC (Z ) \subset \CC (X ) \, ,
$$ 
where the field of rational functions on $Z$ is 
the  field $\CC ( X )^{\langle \s \rangle}$ of invariant functions under $\s$.

Let us recall, following \cite{cyclic}, the structure of cyclic extensions $\CC ( W ) \subset \CC ( V )$,
where $W, V$ are normal varieties. Here the Galois group $G$ is cyclic of order $m$,
 $G \cong \ZZ/m\ZZ$; later we use 
this description with $m=2$ and with $m=n$ to study  $D_n$-field extensions. 
Let $\s \in G$ be a generator and let $\z \in \CC$ be a primitive $m$-th root of the unity.
Then there exists ${v}\in \CC ( V )$ and ${f}\in \CC ( W )$, such that 
\begin{equation}\label{cyclicfe}
\begin{cases}
\CC ( V ) &= \CC ( W ) ({v}) \, ,  \quad {v}^m = {f} \, , \\
\s ( {v} ) &= \z {v} \, . 
\end{cases}
\end{equation}
Concretely, ${v}$ can be chosen to be any non-zero element of the form
$$
{v} = \sum_{i=0}^{m-1} \z^{m-i}\s^i (\tilde{v}) \, , \quad \tilde{v} \in \CC ( V ) \, .
$$
Furthermore, it is possible to choose $v$ and $f$,
such that \eqref{cyclicfe} holds and 
\begin{equation}\label{cyclicfe1}
f = \prod_{i=1}^{m-1} (\d_i)^i \, ,  
\end{equation} 
where $\d_1 , \ldots , \d_{m-1}$ are regular sections of invertible sheaves 
on $W\setminus \sing (W)$.
To see this, let $\hat{v}$ and $\hat{f}$ satisfying \eqref{cyclicfe}, and
consider the Weil divisor associated to $\hat{f}$,
$( \hat{f} ) = \sum_{U} v_U ( \hat{f} ) U$, where $U\subset W$ varies among the prime 
 divisors of $W$ and $v_U$ is the valuation of $U$. 
 On the non-singular locus $W\setminus \sing (W)$, $U\cap (W\setminus \sing (W))$ is an
 effective  Cartier divisor,
hence $U\cap (W\setminus \sing (W)) = \{ \d_U =0\}$, where $\d_U$ is a 
regular section of an invertible sheaf on $W\setminus \sing (W)$, hence
$$
\hat{f} = \prod_U \d_U^{v_U ( \hat{f} )} \, .
$$ 
Let now $v_U ( \hat{f} ) = mq_U ( \hat{f} ) + r_U ( \hat{f} )$, where  $q_U ( \hat{f} ) , 
r_U ( \hat{f} ) \in \ZZ$,
$0\leq r_U ( \hat{f} ) < m$, 
and define 
$$
\d_i = \prod_{r_U(\hat{f})=i} \d_U  \, , \quad i=1 , \ldots, m-1 \, ;
$$ 
then the claim holds true with $v:= \hat{v}\prod_U(\d_U)^{-q_U ( \hat{f})}$.

\begin{rem}\label{dicyclic}
If $W$ is  factorial, the group of Weil divisors coincides with that 
of Cartier divisors, hence $\d_i$ corresponds to a Weil divisor $D_i$
which is reduced but not necessarily irreducible. Geometrically $D_i$ is the divisorial part of the branch
locus $\sB =\sum_{i=1}^{m-1}D_i$ where the local monodromy is $\s^i$ and 
$v$ is a rational section of a line bundle $L$ on $W$ which  satisfies the linear equation
\begin{equation}\label{dicycliceq}
mL \equiv \sum_{i=1}^{m-1} iD_i \, .
\end{equation}
Conversely, one can construct in a natural way a $\ZZ/m\ZZ$-cover
starting from a line bundle $L$ and effective reduced divisors without common components  $D_1, \ldots , D_{m-1}$,
such that \eqref{dicycliceq} holds (\cite{pardini}, \cite{cyclic}).
\end{rem}

The following proposition describes dihedral field extensions (see also \cite{Tok}).
\begin{prop}\label{dfe}
Let $\CC ( Y ) \subset \CC ( X )$ be a $D_n$-extension. Then 
there exist $a, F \in \CC ( Y )$ and $x\in \CC ( X )$, such that 
$$
\CC ( X ) = \CC ( Y ) (x) \, , \qquad x^{2n}-2ax^n + F^n = 0 \, ,
$$
 with $D_n$-action  given as follows: $\s (x) = \z x$, $\tau ( x ) = F/x$, where $\z \in \CC$ is a primitive 
$n$-th root of $1$.

Conversely, given $a, F \in \CC ( Y )$, such that $x^{2n} - 2ax^n + F^n \in \CC ( Y ) [x]$
is irreducible, then $\frac{\CC ( Y ) [ x ]}{( x^{2n} - 2ax^n + F^n )}$ is a $D_n$-field extension 
of $\CC ( Y )$ with $D_n$-action given as before. Hence the normalization 
of $Y$ in $\frac{\CC ( Y ) [ x ]}{( x^{2n} - 2ax^n + F^n )}$ is a $D_n$-covering of $Y$.
\end{prop}
\begin{proof}
Consider the quotient $Z:= X/\langle \s \rangle$ of $X$ by the 
 cyclic subgroup $\langle \s \rangle \cong \ZZ/n\ZZ$ and 
let $q\colon Z \ra Y$ be the induced double cover. 
From the previous description of  cyclic field extensions, we have that
\begin{equation*}
\begin{cases}
\CC ( Z ) &= \CC ( Y ) (z) \, , \quad z^2 = f \in \CC ( Y ) \\
\bar{\tau} (z) &= -z
\end{cases} 
\end{equation*}
where $\bar{\tau} \in D_n/ \langle \s \rangle$ is the class of $\tau$.
Since $Y$ is factorial we can assume that $f$ is a regular section
of an invertible sheaf on $Y$. \\
Now consider the extension 
$\CC ( Z ) \subset \CC ( X )$. Let $x\in \CC ( X )$,  $g\in \CC ( Z )$, such that 
$$
\begin{cases}
\CC ( X ) &= \CC ( Z ) ( x ) \, , \quad x^n = g  \, , \\ 
\s ( x ) &= \z x \, , 
\end{cases}
$$
where $\z \in \CC$ is a primitive $n$-th root of $1$, and notice that 
$g \in \CC ( Z )$ can be written as
$$
g=  a + zb \, , \qquad \mbox{with} \quad a, b \in \CC ( Y ) \, .
$$ 
Without loss of generality $b=1$ in the previous formula,
otherwise we replace $f$ with $b^2f$ and $z$ with $bz$.
Moreover, $x\tau ( x )$ is invariant under the action of $D_n$, therefore 
$$
x\tau ( x ) = F \in \CC ( Y ) \, , \quad \mbox{equivalently} \quad \tau ( x ) = F/x \, .
$$
Finally, $F^n = x^n \tau (x^n) = g\bar{\tau}( g ) = (a+z)(a-z) = a^2 - f$, hence
$$
f=  a^2 - F^n  \, .
$$
To conclude, we observe that $z\in \CC ( Y ) (x)$ because $x^n = a+ z$,
so $\CC ( X ) = \CC ( Y ) (x)$ and by   construction it follows that
$$
x^{2n} - 2ax^n + F^n = 0 \, . 
$$

For  the last statement, notice that the field extension
$$
\CC ( Y ) \subset \frac{\CC ( Y )[ x ]}{(x^{2n} - 2ax^n + F^n)}
$$
is Galois with group $D_n$, indeed  the conjugates 
of $x$, namely $\z^i x$ and $ \z^iF/x$, for $i=0, \ldots , n-1$, belong to 
$\frac{\CC ( Y )[ x ]}{(x^{2n} - 2ax^n + F^n)}$.
\end{proof}

\subsection{Structure of $D_n$-covers}
Let  $\pi \colon X \ra Y$ be a flat $D_n$-cover with $Y$ smooth.
Then $\pi_* \Oh_X$
is a locally free sheaf of $\Oh_Y$-modules, i.e. a vector bundle on $Y$. 
Recall that the sheaf $\pi_* \Oh_X$ carries  a natural structure of 
$\Oh_Y$-algebras, which is given by the  product of regular functions on $X$, 
\begin{equation}\label{m}
m \colon \pi_* \Oh_X \otimes_{\Oh_Y} \pi_* \Oh_X \ra \pi_* \Oh_X \, ;
\end{equation}
the $D_n$-action on $X$ gives to $\pi_* \Oh_X$ the structure of a $D_n$-sheaf (see below).
On the other hand, the variety $X$  is completely determined by $\pi_* \Oh_X$ and $m$, since $X= \spec  (\pi_* \Oh_X)$
(\cite[II, 5.17]{Hart}); the $D_n$-action on $X$ is given by
the structure of  $D_n$-sheaf on $\pi_* \Oh_X$.

Recall that, for any finite group $G$ the \textit{regular representation} of $G$, $\bC[G]$, is the vector space 
with a basis $\{ e_g\}_{g\in G}$ indexed by the elements of $G$; for any $h\in G$, an  
endomorphism of $\bC[G]$ is defined by $e_g\mapsto e_{hg}$, for all $g\in G$.
In a similar way one defines a sheaf of $\Oh_Y$-algebras $\Oh_Y[G]$, for any variety $Y$. 
A sheaf  $\mathcal{F}$  of $\Oh_Y$-modules is a {\bf $G$-sheaf}, if it has a structure of 
sheaf of $\Oh_Y[G]$-modules. If moreover $\mathcal{F}$ is a vector bundle, then  its fibers 
carry a linear $G$-action and so we can see $\mathcal{F}$ as a family of representations of $G$ parametrized by  $Y$.

For any representation $\rho\colon G\ra {\rm GL}(V)$ of $G$, its {\it canonical decomposition} (see \cite[\S 2.6]{Serre}) is
the unique decomposition 
$$
V=V_1\op \ldots \op V_N \, 
$$
defined as follows. Let $W_1 , \ldots , W_N$ be the different  irreducible representations  of $G$.
Then each $V_i$ is the direct sum of all the  irreducible representations of $G$ in $V$ that are  isomorphic to $W_i$. 
If $\chi_1 , \ldots , \chi_N$ are the characters of $W_1 , \ldots , W_N$,   and 
$n_i = \dim W_i$, then
\begin{equation}\label{candec}
p_i = \frac{n_i}{|G|} \sum_{g\in G}\overline{\chi_i (g)}\rho (g) \in {\rm End}(V)
\end{equation}
is the projection of $V$ onto $V_i$, for any $i=1 , \ldots , N$, where $\overline{\chi_i (g)}$ is the complex-conjugate of 
$\chi_i (g)$.

Let now $\mathcal{F}$ be a locally free $G$-sheaf on $Y$ 
with action $\rho \colon G \ra {\rm GL}_{\Oh_Y}(\mathcal{F})$. 
Via  \eqref{candec} we define an endomorphism 
$p_i \in {\rm End}_{\Oh_Y}(\mathcal{F})$, for any $i=1, \ldots , N$. Setting $\mathcal{F}_i:= {\rm Im}(p_i)$, we have the 
following decomposition:
$$
\mathcal{F} = \mathcal{F}_1 \oplus \ldots \oplus  \mathcal{F}_N \, .
$$  
For any $i=1, \ldots , N$, $\mathcal{F}_i$ is   the \textit{eigensheaf} of $\mathcal{F}$
corresponding to the (irreducible) representation with character $\chi_i$.  
Notice that $\mathcal{F}_i$ is a vector sub-bundle of  $\mathcal{F}$, for all $i$.
In particular, when $\mathcal{F}= \pi_* \Oh_X$,   $\pi \colon X \to Y$ is a flat $G$-cover,
we have that  $\pi_* \Oh_X$ is a locally free sheaf of $\Oh_Y[G]$-modules of rank one,
i.e. the fibres of  $\pi_* \Oh_X$ are isomorphic to $\bC [G]$ as $G$-representations. 
Indeed, the previous procedure gives the decomposition  
$\pi_* \Oh_X = \left( \pi_* \Oh_X\right)_1 \oplus \ldots \oplus  \left( \pi_* \Oh_X\right)_N$, with 
$\left( \pi_* \Oh_X\right)_i \subset  \pi_* \Oh_X$ a sub-bundle, for any $i$. 
By construction, the fibres of $\left( \pi_* \Oh_X\right)_i$ are  isomorphic to each other as $G$-representations.
So, it is enough to consider the restriction of $\pi_* \Oh_X$ on the complement $Y\setminus \sB$
of the branch divisor $\sB$, where the assertion follows easily.  

In order to describe the sheaves $\left( \pi_* \Oh_X\right)_i$,  $i=1, \ldots , N$, when
$\pi \colon X \to Y$ is a flat $D_n$-cover,  let us briefly 
recall the representation theory of the dihedral groups. They depend on the parity of $n$.

\medskip

\noindent \underline{$n$ odd.} There are two irreducible representations of degree $1$
with characters $\chi_1$ and $\chi_2$,
\begin{eqnarray*}
\begin{array}{ l | c | r }
     & \s^k & \s^k\tau \\ \hline
    \chi_1 & 1 & 1 \\ \hline
    \chi_2 & 1 & -1 \\
\end{array} \, ,
\end{eqnarray*}
and $\frac{n-1}{2}$ irreducible representations of degree $2$,
\begin{eqnarray}\label{roh}
\rho^\ell (\s^k) = \left( \begin{matrix} \z^{k\ell} & 0 \\ 0 & \z^{-k\ell} \end{matrix} \right) \, ,
\qquad 
\rho^\ell (\s^k \tau) = \left( \begin{matrix} 0 & \z^{k\ell} \\  \z^{-k\ell} & 0 \end{matrix} \right) \, ,
\end{eqnarray}
where $\z \in \CC^*$ is a primitive $n$-th root of $1$, $1\leq \ell \leq \frac{n-1}{2}$, $k=0,\ldots , n-1$.

\medskip

\noindent \underline{$n$ even.} In this case there are $4$ representations
of degree $1$ with characters $\chi_1, \chi_2, \chi_3$  and $\chi_4$,
\begin{eqnarray*}
\begin{array}{ l | c | r }
     & \s^k & \s^k\tau \\ \hline
    \chi_1 & 1 & 1 \\ \hline
    \chi_2 & 1 & -1 \\ \hline
    \chi_3 & (-1)^k & (-1)^k \\ \hline
    \chi_4 & (-1)^k & (-1)^{k+1}
\end{array} \, ,
\end{eqnarray*}
and, for any $1\leq \ell \leq \frac{n}{2} - 1$, the irreducible representation $\rho^\ell$
defined by \eqref{roh}. 

\medskip

As a consequence   we have that 
$$
\pi_* \Oh_X = \Oh_Y \op \sL \op_{\ell=1}^{\frac{n-1}{2}}  \left( \pi_* \Oh_X \right)_\ell \, , \quad   \quad \mbox{if $n$ is odd} \, ,
$$
 and
$$
\pi_* \Oh_X = \Oh_Y \op \sL \op \MMM \op
\mathcal{N} \op_{\ell=1}^{\frac{n}{2} - 1}  \left( \pi_* \Oh_X \right)_\ell \, , \quad \quad \mbox{if $n$ is even} \, ,
$$
where $\sL$, $\MMM$, $\mathcal{N}$ are the line bundles
corresponding  to the $1$-dimensional representations with characters  respectively $ \chi_2, \chi_3, \chi_4$.

Notice that the sections of $\sL$ are invariant under the action of the 
subgroup  $\langle \s \rangle \cong \ZZ/n\ZZ \leq D_n$, hence they are regular functions 
on $Z$.
Moreover we have:
\begin{equation}\label{p*oz}
q_* \Oh_Z = \Oh_Y \op \sL \, , \quad \sL^{\otimes 2} \cong \Oh_Y (- \sB_q) \, , 
\end{equation}
where $\sB_q \subset Y$ is the branch divisor of $q$.

If $n$ is even, the line bundles $\MMM$ and $\mathcal{N}$ have a similar interpretation, they arise from
the $(\bZ/2\bZ \times \bZ/2\bZ)$-cover $X/\langle \s^2 , \tau \rangle \to Y$.

In order to get further information on the rank $4$ vector bundles $\left( \pi_* \Oh_X \right)_\ell$,
we assume that $X$ is normal and use the factorization $\pi = q\circ p$ and the theory of cyclic covers.
Let us denote with  $Z^0$  the smooth locus of $Z$, $X^0 = p^{-1}(Z^0)$ and $p^0 \colon X^0 \to Z^0$ be the restriction 
of $p$. 
From the structure theorem for 
cyclic covers (\cite{pardini}, \cite{BC08},  \cite{cyclic}) 
it follows that $p^0$ is determined by divisor classes 
$L_1 , \ldots , L_{n-1}$ and reduced effective divisors $D_1^0 , \ldots , D_{n-1}^0 \subset Z^0$
without common components, such that 
\begin{equation}\label{cyclic}
L_i + L_j\equiv L_{\overline{i+j}}  - \sum_{k=1}^{n-1} \e^k_{i,j} D^0_k \, ,
\end{equation}
where $\overline{i+j} \in \{ 0, \ldots , n-1\}$, $\overline{i+j} = i+j \, ({\rm mod} \, n)$, $L_0:= \Oh_{Z^0}$.

Let us briefly recall   the geometric interpretation of the previous data. For any $i=0, \ldots , n-1$,
the line bundle $\Oh(L_i)$ is the subsheaf of $(p^0)_* \Oh_{X^0}$ consisting of the regular functions $f$
on $X^0$ such that $\s^*f = \exp \left( \frac{2\pi \sqrt{-1}}{n}i \right) f$. For any $k=1, \ldots , n-1$,
the divisor  $D_k^0\subset Z^0$ is the  union of the components $\D$ of the branch divisor $\sB_{p^0}$ of $p^0$
such that, for any component $T\subset (p^0)^{-1}(\D)$, the stabilizer of the generic point of $T$
is the cyclic subgroup of $\langle \s \rangle$ generated by $\s^k$ and there is a uniformizing parameter
$x \in \Oh_{X^0,T}$ such that $(\s^k)^* x = \exp \left( \frac{2\pi \sqrt{-1}}{|\langle \s^k \rangle|}\right)x$,
where $|\langle \s^k \rangle|$ is the order of $\s^k$. For every $i,j = 0 , \ldots , n-1$, $\e^k_{i,j}$
is defined as follows: let $\imath_i(k), \imath_j(k) \in \{ 0, \ldots , |\langle \s^k \rangle| -1 \}$
be such that 
\begin{eqnarray*}
\exp \left( \frac{2\pi \sqrt{-1}}{n} ik\right) &=& \exp \left( \frac{2\pi \sqrt{-1}}{|\langle \s^k \rangle|}\right)^{\imath_i(k)} \, , \\
\exp \left( \frac{2\pi \sqrt{-1}}{n} jk\right) &=& \exp \left( \frac{2\pi \sqrt{-1}}{|\langle \s^k \rangle|}\right)^{\imath_j(k)}
\end{eqnarray*}
respectively, then 
\begin{equation}\label{eijk}
\e^k_{i,j} =
\begin{cases}
1 & , \quad {\rm if} \quad \imath_i(k) + \imath_j(k) \geq |\langle \s^k \rangle| \\
0 & , \quad {\rm otherwise.}
\end{cases}
\end{equation}

\begin{prop}\label{Fi}
Let $Y$ be a smooth variety and let $\pi \colon X \to Y$ be a flat $D_n$-cover
with $X$ normal.
Let $p\colon X \to Z$ and $q \colon Z \to Y$
be the intermediate cyclic covers defined previously. Then the following holds true.
\begin{itemize}
\item[(i)]
$p_* \Oh_X =  \op_{i=0}^{n-1} \sF_i$, where $\sF_0  = \Oh_Z$ and, 
for any $i=1, \ldots , n-1$, $\sF_i$ is a divisorial sheaf on $Z$. 
For any $i,j =1, \ldots , n-1$ the product \eqref{m} induces  an isomorphism as follows:
\begin{equation}\label{?}
\left( \sF_i \ot_{\Oh_Z} \sF_j \right)^{**} \cong \left( \sF_{\overline{i+j}} \ot_{\Oh_Z} 
\Oh_Z \left( - \sum_{k=1}^{n-1} \e^k_{i,j} D_k \right)  \right)^{**} \, ,
\end{equation}
where $\overline{i+j} \in \{ 0, \ldots , n-1\}$, $\overline{i+j} = i+j \, ({\rm mod} \, n)$, 
  for any $k=1, \ldots , n-1$,  $D_k = \overline{D_k^0}$ is the closure of the divisor 
$D_k^0$ in \eqref{cyclic}, and $()^{**}$ is the double dual in the category of $\Oh_Z$-modules. 
\item[(ii)]
For any $i=1, \ldots , n-1$,  $U_i := q_* ( \sF_i)$ 
is a vector bundle of rank $2$ on $Y$, in particular $\sF_i$ is flat over $Y$. Furthermore,  
$$
\left( \pi_* \Oh_X \right)_\ell =  U_\ell \op U_{n-\ell} \, ,
$$
where $\ell = 1, \ldots , \frac{n-1}{2}$, if $n$ is odd, and $\ell = 1, \ldots , \frac{n}{2} -1$ 
 if $n$ is even; in the case where $n$ is even, $U_{\frac{n}{2}}= \mathcal{M} \op \mathcal{N}$.
\end{itemize}
\end{prop}
\begin{proof}
Consider the following Cartesian diagram:
\begin{equation}\label{piq}
\begin{CD}
X^0 @>{\imath}>> X  \\
@V{p^0}VV @VVpV\\
Z^0 @>{\jmath}>> Z 
\end{CD}
\end{equation} 
where $Z^0$ is the smooth locus of $Z$, $\jmath \colon Z^0 \to Z$ is the inclusion, $X^0=p^{-1}(Z^0)$,
$\imath \colon X^0 \to X$ is the inclusion, and $p^0 = p_{|X^0}$.
Notice that under our hypotheses $Z$ is normal. 

Since $\Oh_X= \imath_* \Oh_{X^0}$ and $\Oh_Z = \jmath_* \Oh_{Z^0}$ (see, e.g. \cite{ReidC3f}),
\begin{eqnarray}\label{q*ox}
p_* \Oh_X &=& p_* \imath_* \Oh_{X^0} = (p\circ \imath)_*\Oh_{X^0} = \nonumber\\
&=& \jmath_* p^0_* \Oh_{X^0} = \jmath_* \left( \Oh_{Z^0} \op_{i=1}^{n-1} \Oh_{Z^0}(L_i) \right)  =\\
&=& \Oh_{Z}  \op_{i=1}^{n-1}  \jmath_* \Oh_{Z^0}(L_i) \, . \nonumber
\end{eqnarray}
Let us define $\mathcal{F}_i := \jmath_* \Oh_{Z^0}(L_i)$, for any $i=1, \ldots , n-1$.
The product \eqref{m} gives morphisms $m_{ij} \colon \sF_i \ot \sF_j \to \sF_{\overline{i+j}}$.
Since $\sF_{\overline{i+j}}$ is reflexive, $m_{ij}$ is determined by its restriction 
on the smooth locus $Z^0$. On $Z^0$, by \eqref{cyclic}, $m_{ij}$ gives an isomorphism 
$\Oh_{Z^0}(L_i + L_j) \cong \Oh_{Z^0}(L_{\overline{i+j}} - \sum_{k=1}^{n-1}D_k^0)$.
This implies \eqref{?}.

To prove (ii), apply $q_*$ to  \eqref{q*ox} and use the following equalities:
$\pi = q\circ p$,  $q_* \Oh_Z = \Oh_Y \op \mathcal{L}$. Notice that,  $q_* \mathcal{F}_i$
is locally free for any $i=0, \ldots , n-1$,  since $\pi_* \Oh_X$ is locally free 
and $\pi_* \Oh_X = \op_{i=0}^{n-1} q_* \sF_i$. The equality $\left( \pi_* \Oh_X \right)_\ell = U_\ell \op U_{n-\ell}$
follows directly from the definition of $\left( \pi_* \Oh_X \right)_\ell
\subset \pi_* \Oh_X$ as the eigensheaf corresponding to the representation $\rho^\ell$.
Finally, in the case where $n$ is even, $\tau^*$ acts on   $U_{\frac{n}{2}}$; 
$\sM$ is the $\tau^*$-invariant subsheaf, while $\sN$ is the  $\tau^*$-anti-invariant one.
\end{proof}

Notice that \eqref{?} implies the following relation between $\sF_1$
and the Weil divisors $D_1, \ldots , D_{n-1}$,
$$
\left( \sF_1^{\ot n } \right)^{**} \cong \Oh_Z (-\sum_{k=1}^{n-1} k D_k) \, ,
$$
which is the analogous 
of Remark \ref{dicyclic} and equation  \eqref{dicycliceq} when the base of the cover is normal.

For later use we observe that the branch divisor $\sB_{p^0}=\sum_{k=1}^{n-1} D_k^0$ of $p^0 \colon X^0 \to Z^0$
is invariant under the involution $\overline{\tau} \colon Z^0 \to Z^0$ induced by $\tau$,
that is $\overline{\tau}^* \sB_{p^0} = \sB_{p^0}$, since $\s \tau = \tau \s^{-1}$.
So, there exists an effective Cartier divisor $\D_{p^0} \subset Y^0$, such that
$(q^0)^* \D_{p^0} = \sB_{p^0}$, where  $q^0 := q_{|Z^0}$. Let us define 
\begin{equation}\label{Bq}
\D_p := \overline{\D_{p^0}} \in {\rm Div} ( Y ) \, .
\end{equation}

\begin{rem} 
Notice that, for any $i=1, \ldots , n-1$, $\overline{\tau}(D_i) = D_{n-i}$, so
$\overline{\tau}(D_i)$ does not have any common component with $D_i$
for any $i<n/2$, while, if $n$ is even, $\overline{\tau}(D_{\frac{n}{2}}) = D_{\frac{n}{2}}$.
\end{rem}

\medskip

In the remaining of the section we determine building data for $D_n$-covers of $Y$.
To this aim we first  derive some properties of the vector bundles $U_1 , \ldots , U_{n-1}$
(\textbf{1.}--\textbf{3.} below).

\medskip 
 
\noindent \textbf{1.} For any $i=1, \ldots , n-1$, the involution $\tau \colon X \to X$ induces an isomorphism $\tau^* \colon U_i \to U_{n-i}$,
 such that $\tau^* \circ \tau^* $ is the identity.\\
\begin{proof} For any open  $V\subset Y$, $U_i (V):=\mathcal{F}_i (q^{-1}(V))$ consists of the
regular functions $f\in \Oh_X(\pi^{-1}(V))$ such that $\s^* f = \exp (\frac{2\pi \sqrt{-1}}{n}i) f$.
From the relation $\s \tau = \tau \s^{-1}$ it follows that $f\mapsto \tau^* f$ gives a morphism 
$\tau^* \colon U_i (V) \to U_{n-i} (V)$ of $\Oh_Y$-modules.
Since $\tau^2 =1$, $\tau^* \circ \tau^* = {\rm Id}$ and $\tau^*$ is an isomorphism.  \end{proof}

\noindent \textbf{2.} For any $i=1, \ldots , n-1$, $U_i$ has a structure of $\sR$-module
given by the restriction $m_i \colon \mathcal{L} \otimes U_i \to U_i$ of the product \eqref{m},
where $\sR = \Oh_Y \op \sL = q_* \Oh_Z$ (see Section 3.). In particular $m_i^2 = \d_q \cdot \operatorname{Id}_{U_i}$,
where $\d_q \in H^0 (Y, \Oh_Y (\sB_q))$ is such that $\sB_q = \{ \d_q =0 \}$
($\d_q =F$ of Section 3).
Furthermore, the isomorphism $\tau^* \colon U_i \to U_{n-i}$ allows us to identify 
$m_{n-i}$ with $-m_i$, as it follows from the commutativity of the following diagram
and the fact that $m_i \circ (\tau^* \ot {\rm Id}_{U_i}) = -m_i$
\begin{equation}\label{bf2}
\begin{CD}
\mathcal{L} \ot U_i @>{m_i \circ (\tau^* \ot {\rm Id}_{U_i})}>> U_i \\
@V{{\rm Id}_{\mathcal{L}} \ot \tau^*}VV @VV{\tau^*}V \\
\mathcal{L} \ot U_{n-i} @>{m_{n-i}}>> U_{n-i} 
\end{CD} 
\end{equation}

\medskip

\noindent \textbf{3.} For any $i,j =1, \ldots , n-1$, the product \eqref{m} induces, by restriction,
a morphism
$$
m_{ij} \colon U_i \ot U_j \to U_{\overline{i+j}} \, , 
$$
where $\overline{i+j} \in \{ 0, \ldots , n-1\}$, $ \overline{i+j} = i+j$ (mod $n$), and by definition
$U_0 = q_* \Oh_Z$. In particular, for any $i=1, \ldots , n-1$,
there is a morphism
$$
m_{i,n-i} \colon U_i \ot U_{n-i} \to \Oh_Y \op \mathcal{L} \, .
$$
In the following, we will denote again with $m_{i,n-i}$ the previous morphism under the identification 
$\tau^* \colon U_{n-i} \to U_i$, hence
$$
m_{i,n-i} \colon U_i \ot U_{i} \to \Oh_Y \op \mathcal{L} \, .
$$

Let $m_{i,n-i}^+ \colon U_i \ot U_{i} \to \Oh_Y$ and $m_{i,n-i}^- \colon U_i \ot U_{i} \to \mathcal{L}$
be the compositions of $m_{i,n-i}$ with the projections onto $\Oh_Y$ and  $\mathcal{L}$ respectively.
Notice that, with respect to the involution on $U_i \ot U_{i}$ that exchanges 
the factors, $m_{i,n-i}^+$ is symmetric, while $m_{i,n-i}^-$ is antisymmetric, hence
they can be seen as morphisms
\begin{eqnarray}\label{min-ipm}
 m_{i,n-i}^+ \colon {\rm Sym}^2 (U_i ) \to \Oh_Y \, ,  \qquad 
  m_{i,n-i}^- \colon \we^2 (U_i) \to \mathcal{L} \, ,
\end{eqnarray} 
or equivalently as sections
\begin{eqnarray}\label{min-ipmsec}
 m_{i,n-i}^+ \in H^0 (Y, {\rm Sym}^2 (U_i^\vee )) \, ,  \,
  m_{i,n-i}^- \in H^0(Y, \we^2 (U_i^\vee) \ot \mathcal{L}) \, .
\end{eqnarray} 

\begin{prop}\label{U3}
For any $i=1, \ldots , n-1$ the following statements hold true.
\begin{itemize}
\item[(i)]
$m_{i,n-i}^+$ is determined by $m_i$ and $m_{i,n-i}^-$.
\item[(ii)]
The divisor  of zeros of $m_{i,n-i}^-$ coincides  with  the divisor $\D_p$
defined in \eqref{Bq}. In particular  $m_{i,n-i}^-$ yields 
an isomorphism between $\we^2 U_i$ and $\mathcal{L} \ot \Oh_Y (-\D_p)$.
\end{itemize}
\end{prop}
\begin{proof}
(i) Let $y\in Y$, let $(U_i)_y$ be the stalk of $U_i$ over $y$
 and let $s_1, s_2 \in (U_i)_y$. Then
$$
m_{i,n-i} (s_1\ot s_2) = s_1 \tau^* (s_2) \, ,
$$
hence
\begin{eqnarray*}
m_{i,n-i}^+ (s_1\ot s_2) &=& \frac{1}{2} \left(s_1 \tau^* (s_2) + \tau^* (s_1) s_2 \right) \, ,\\
m_{i,n-i}^- (s_1\ot s_2) &=& \frac{1}{2} \left(s_1 \tau^* (s_2) - \tau^* (s_1) s_2 \right) \, ; 
\end{eqnarray*}
where the product $s_1 \tau^* (s_2)$ (respectively $\tau^* (s_1) s_2 $) is the 
usual one between stalks of regular functions defined in some neighborhood 
of $\pi^{-1}(y)$. For any $r\in (\mathcal{L})_y$,  the associativity of the multiplication implies that
$$
m_{i,n-i} \left( m_i (r\ot s_1) \ot s_2 \right) = m \left( r \ot m_{i,n-i} (s_1 \ot s_2) \right)
$$
and hence
$$
m_{i,n-i}^\pm \left( m_i (r\ot s_1) \ot s_2 \right) = m \left( r \ot m_{i,n-i}^\mp (s_1 \ot s_2) \right) \, .
$$
This implies that, under the natural identification $\Oh_{Y,y} \cong {\rm End}((\mathcal{L})_y)$,
$m_{i,n-i}^+(s_1 \ot s_2) = m_{i,n-i}^- (m_i( (\_) \ot s_1) \ot s_2)$, and hence the claim follows.

\medskip

\noindent (ii) Without loss of generality we   assume that $Z$ is smooth.
Indeed, since $Z$ is normal,
its singular locus ${\rm Sing}(Z)$ has codimension $\geq 2$,
so, for $Z^0 := Z\setminus {\rm Sing}(Z)$ and $Y^0 := q(Z^0)$,  the divisor of zeros of $m_{i,n-i}^-$ and $\D_p$ are determined by their restrictions
to $Y^0$.

For any $y\in Y$,
let us consider the stalk $(m_{i,n-i}^-)_y$ of $m_{i,n-i}^-$ at $y$.
Let $u_i , v_i$ be  a  basis   of  $(U_i)_y$ as $\Oh_{Y,y}$-module,  
then $u_i \ot v_i - v_i \ot u_i$ is a  basis of $(\we^2 U_i)_y$, viewed as the submodule 
of  $U_i \ot U_i$. We have:
\begin{eqnarray}\label{min-i-uv}
m_{i,n-i}^- (u_i \ot v_i - v_i \ot u_i) &=& u_i \tau^*(v_i) - v_i \tau^*(u_i) \nonumber \\
&=& u_i \tau^*(v_i) - \tau^*(u_i \tau^*(v_i))  \, ,
\end{eqnarray}
where we consider $u_i$ and $v_i$ as stalks of regular functions  
defined on some open neighbourhood 
of $\pi^{-1}(y)$ in $X$,
such that $\s^* (u_i) = \exp(\frac{2\pi\sqrt{-1} }{n}i) u_i$
and $\s^* (v_i) = \exp(\frac{2\pi\sqrt{-1} }{n}i) v_i$. 

Let us choose local analytic coordinates $(y_1 , \ldots , y_d)$ for $Y$ at $y$,
and $w$ on $Z$, such that  $Z$ is given locally by the equation $w^2=y_1$.
Furthermore, for any $i=1, \ldots , n-1$, let $e_i$ be a 
basis of $\mathcal{F}_i$ as $\left(\frac{\bC[y_1, \ldots , y_d, w]}{(w^2 - y_1)}
\right)$-module. Then let 
\begin{equation}\label{uivi}
u_i := 1\cdot e_i \, , \quad v_i := w\cdot e_i \, .
\end{equation}
Notice that we can choose the $e_i$'s in such a way that
$\tau^* (e_i ) = e_{n-i}$, for any $i= 1, \ldots , n-1$, since $\tau^*$ identifies 
$\mathcal{F}_i$ with $\mathcal{F}_{n-i}$.
So, substituting \eqref{uivi} in \eqref{min-i-uv} and using 
the equations $\tau^*(w)=-w$ and $\tau^*(e_i)=e_{n-i}$, we obtain:
\begin{eqnarray}\label{eien-i=dq}
m_{i,n-i}^- (u_i \ot v_i - v_i \ot u_i) &=& u_i \tau^*(v_i) - \tau^*(u_i \tau^*(v_i))
\nonumber \\
&=& e_i\tau^*(we_i) - \tau^*(e_i\tau^*(we_i)) \nonumber \\
&=& e_i(-we_{n-i}) - \tau^*(e_i(-we_{n-i})) \nonumber \\
&=& - we_ie_{n-i} - we_ie_{n-i} \nonumber \\
&=& -2we_ie_{n-i} = -2 wb_p \, ,
\end{eqnarray}
where  we have used the fact that
$e_ie_{n-i}= b_p$,  $b_p$ is a local equation for $\sB_p$
(this follows from \eqref{cyclic}).
Finally, let $\d_p$ be a local equation for $\D_p$ such that 
$q^*(\d_p) = b_p$, then the previous computation shows that
$$
m_{i,n-i}^- (u_i \ot v_i - v_i \ot u_i) = - 2wq^*(\d_p) \, ,
$$
from which the claim follows.
\end{proof}

The following result  is a converse to Proposition \ref{Fi}.
\begin{theo}[Structure of $D_n$-covers]\label{STDn}
Let $Y$ be a smooth variety and $n$ be a positive integer.
Then, to the following data a), b) and c), we can associate a $D_n$-cover 
$\pi \colon X \to Y$ in a natural way.
\begin{itemize}
\item[a)]
A line bundle $\mathcal{L}$  and an effective reduced divisor $\sB_q$ on $Y$,
such that  $\mathcal{L}^{\ot 2} \cong \Oh_Y (-\sB_q)$.
\item[b)]
Reduced effective Weil divisors $D_1 , \ldots , D_{\lfloor \frac{n}{2} \rfloor}$ on 
$Z:=\spec (\Oh_Y \op \sL)$ without common components, such that
$\overline{\tau}(D_1 \cup \ldots \cup D_{\lfloor \frac{n-1}{2}\rfloor})$ 
doesn't have common components with 
$D_1 \cup \ldots \cup D_{\lfloor \frac{n-1}{2}\rfloor}$, 
and, in the case where  $n$ is even,  
$\overline{\tau}(D_{\frac{n}{2}}) = D_{\frac{n}{2}}$;
where $\lfloor a \rfloor$ denotes the integer part of
a number  $a\in \QQ$, and 
$\overline{\tau}$ is the involution of the  double cover $q\colon Z \to Y$.
\item[c)]
Divisorial sheaves $\sF_1 , \ldots , \sF_{\lfloor \frac{n}{2}\rfloor}$ on  $Z$
flat over $\Oh_Y$, such that,
 for any $i,j = 1, \ldots , n-1$,
\eqref{?} holds, and if $n$ is even
$\sF_{\frac{n}{2}} = \overline{\tau}^*( \sF_{\frac{n}{2}})$; where for $\lfloor \frac{n}{2} \rfloor < \ell , k \leq n-1$,
$\sF_\ell := \overline{\tau}^* ( \sF_{n-\ell})$ and $D_k := \overline{\tau}(D_{n-k})$,
 the coefficients $\e^k_{ij}$ are defined in \eqref{eijk}.
\end{itemize}

The variety $X$ so constructed is normal if and only if, 
setting $\varkappa:= \gcd \{ k=1, \ldots , n-1 \, | \, D_k \not=0 \}$,
then either $\varkappa=1$, or $\frac{n}{\varkappa}\sF_1 - \sum_{k=1}^{n-1}\frac{k}{\varkappa}D_k$
has order precisely $\varkappa$ in the group of divisorial sheaves of $Z$. 
In this case $\pi_* \Oh_X= \Oh_Y \op \sL \op_{i=1}^{{\lfloor n/2\rfloor}} q_*\sF_i$, in particular $\pi$ is flat.
\end{theo}

Before giving the proof of the previous theorem, two remarks are in order.
\begin{rem}
{\bf 1.} Using the results of Section 3, the  divisorial sheaves in c) correspond to 
pairs $(U_1, m_1), \ldots , (U_{\lfloor \frac{n}{2}\rfloor}, m_{\lfloor \frac{n}{2}\rfloor})$
consisting of rank $2$ vector bundles on $Y$,
$U_1, \ldots , U_{\lfloor \frac{n}{2}\rfloor}$,  and morphisms 
$m_i \colon U_i \ot \sL \to U_i$ such that $m_i^2 = \d_q \operatorname{Id}_{U_i}$,
where $\sB_q = \{ \d_q = 0 \}$. With this notation, $\overline{\tau}^* \sF_i$
corresponds to $(U_i , -m_i)$, for any $i=1, \ldots , n-1$.
Furthermore, using Proposition \ref{picz}, the relation \eqref{?} 
can be written in terms of the $(U_i , m_i)$'s and the $D_i$'s.

\smallskip

\noindent {\bf 2.} A similar criterion for the existence of
$D_n$-covers $\pi \colon X \to Y$ has been given in \cite{Tok},
where $X$ is constructed  as the normalization of $Y$
in a certain dihedral field extension of $\CC ( Y )$.
Here, using the structure theorem for cyclic covers, we construct
$\pi \colon X \to Y$ explicitly, we hope that this procedure 
could be useful for further investigations of Galois covers.  
\end{rem} 

\begin{proof}
The data a) determines a flat double cover $q\colon Z \to Y$ in the usual way,
$Z:= {\rm Spec}\left( \Oh_Y \op \mathcal{L}\right)$ and $q$ is given by the 
inclusion $\Oh_Y \hookrightarrow \Oh_Y \op \mathcal{L}$ as the first summand. 
Notice that  $Z$ is normal since $\sB_q$ is reduced.

We define $X:= \spec \left( \Oh_Z \op \sF_1 \op \ldots \op \sF_{n-1} \right)$,
where $\Oh_Z \op \sF_1 \op \ldots \op \sF_{n-1}$ is the sheaf of 
$\Oh_Z$-algebras with algebra structure given by the isomorphisms
\eqref{?} in the following way.  Since $\sF_{\overline{i+j}}$ is a divisorial sheaf, 
for $i,j=1, \ldots , n-1$, a morphism $\sF_i \ot \sF_j \to \sF_{\overline{i+j}}$
is uniquely determined by its restriction on the smooth locus $Z^0$ of $Z$.
Let us  fix firstly sections
$\d_k \in H^0(Z, \Oh_Z (D_k))$, such that $D_k =\{ \d_k = 0 \}$ on $Z^0$,
for any $k=1, \ldots , n-1$.
By Lemma \ref{free}, on $Z^0$ the sheaves $\sF_1 , \ldots , \sF_{n-1}$ are locally free, 
so, locally where they are trivial, we choose generators 
$e_1 , \ldots , e_{n-1}$ such that $e_{n-i} = \overline{\tau}^*(e_i)$, for any $i$.
The algebra structure on the restriction of $\Oh_Z \op \sF_1 \op \ldots \op \sF_{n-1}$
on such open subsets is defined as usual by the equations
$$
e_i e_j = e_{\overline{i+j}} \prod_{k=1}^{n-1} \d_k^{\e^k_{ij}} \, , \quad \mbox{for any}
\, i,j=1, \ldots , n-1 \, .
$$
Notice that, if we choose different generators, $\tilde{e}_1, \ldots , \tilde{e}_{n-1}$
satisfying the same conditions $\tilde{e}_{n-i} = \overline{\tau}^*(\tilde{e}_i)$,
then we obtain an algebra canonically isomorphic to the previous one.
Hence the construction globalizes and $\Oh_Z \op \sF_1 \op \ldots \op \sF_{n-1}$
becomes a sheaf of $\Oh_Z$-algebras. The morphism $\pi \colon X \to Y$
is defined as usual.

From the previous construction and from the fact that $D_{n-k}=\overline{\tau}(D_k)$,
it follows that $\overline{\tau}^* \colon \Oh_Z \op \sF_1 \op \ldots \op \sF_{n-1}
\to \Oh_Z \op \sF_1 \op \ldots \op \sF_{n-1}$ is a morphism of $\Oh_Z$-algebras.
This defines an involution  $\tau \colon X \to X$. By construction we have that
$\s^* \circ \tau^* = \tau^* \circ (\s^{-1})^*$, hence $X$ carries an action of $D_n$
such that $\pi \colon X \to Y$ is a $D_n$-cover.

The condition  for $X$ to be normal follows from Theorem 1.1 in \cite{cyclic}.
\end{proof}

\begin{rem}
Similarly as for cyclic covers, $X$ in Theorem \ref{STDn} is determined
by the data a), b) and $\sF_1$ (or $(U_1, m_1)$), such that 
$$
(\sF_1^{\ot n})^{**} \cong \Oh_Z \left(-\sum_{k=1}^{n-1} k D_k \right) \, .
$$
Furthermore, by Proposition \ref{kn}, $(\sF_1^{\ot n})^{**} = \left( {\rm Sym}^n (U_1)
\right)/\sK_n$.
\end{rem}

\subsection{$D_3$-covers and triple covers.}\label{d3&triple}
In this section we  restrict ourselves to $D_3$-covers and  relate the  results that we obtained so far
with the structure theorem for flat finite morphisms of degree $3$
of algebraic varieties  (\cite{miranda}), also called \textit{triple covers}. 
This relation originates from the following  fact: for every 
$D_3$-cover $\pi \colon X \to Y$, any quotient $W:=X/\langle \s^i \tau \rangle$ has a natural structure
of triple cover of $Y$ induced by $\pi$; conversely, if $W\to Y$ is a triple cover which is not Galois,
then its Galois closure is a $D_3$-cover.  Since the elements $\s^i \tau$ are 
pairwise conjugate, the corresponding triple covers are isomorphic, hence in the following
we consider only  $X/\langle  \tau \rangle$.
Notice that the structure of $D_3$-covers has been studied also in \cite{easton}.

Let $\pi \colon X \to Y$ be a $D_3$-cover. We have seen in the previous section that $\pi$
is determined by the locally free sheaves $\mathcal{L} , U_1 , U_2$ and a morphism
of $\Oh_Y$-modules
$$
m\colon \left( \Oh_Y \op \mathcal{L} \op U_1 \op U_2 \right)^{\ot 2} \to \Oh_Y \op \mathcal{L} \op U_1 \op U_2
$$
which gives an associative commutative product. The involution $\tau$ yields an isomorphism $\tau^* \colon U_1 \to U_2$
as explained in the previous section.
Under this identification  one  easily see that $\pi \colon X \to Y$
is determined by $\Oh_Y$, $\mathcal{L}$, $U_1$, $m_{12}$, $m_1$ and $m_{11}$
(see the previous section for their definitions).
In particular,
$$
\Oh_X = \Oh_Y \op \mathcal{L} \op U_1 \op U_1 \, 
$$
and the involution $\tau^*$ acts on the  element $(a,b,c,d) \in \Oh_Y \op \mathcal{L} \op U_1 \op U_1$
as follows:
$$
\tau^* (a,b,c,d) = (a,-b,d,c) \, . 
$$
The  subsheaf of $\tau^*$-invariants, $\left( \Oh_X\right)^{\tau^*} \subset \Oh_X$, consists of the elements 
of the form $(a, 0 , c, c) \in \Oh_Y \op \mathcal{L} \op U_1 \op U_1$, so the sheaf 
of regular functions on $W:= X/\langle \tau \rangle$ is
\begin{equation}\label{OW}
\Oh_W \cong \Oh_Y \op U_1 
\end{equation}
with inclusion in $\Oh_Y \op \mathcal{L} \op U_1 \op U_1$ given by 
$(a,c) \mapsto (a,0,c,c)$.
The morphism $\pi \colon X \to Y$ descends to a morphism $f\colon W \to Y$,
which is a triple cover. We refer to \cite{miranda} for definitions, notations and results concerning
triple covers.

\begin{prop}
Under the identification \eqref{OW}, $U_1$ coincides with the  Tschirnhausen module
of $\Oh_W$ over $\Oh_Y$, that is $U_1$ consists of the elements of $\Oh_W$
whose minimal cubic polynomial   has no square term ($U_1 = E$ in the 
notation of \cite{miranda}).  
Furthermore, the tensor $\phi_2$ in \cite{miranda} coincides with $m_{11} \colon U_1 \ot U_1 \to U_1$,
which has been defined in the previous section.
\end{prop}
\begin{proof}
For $(0,c)\in \Oh_Y \op U_1$, we calculate its cubic power $(0,c)^3$. 
To this aim, we consider its image $(0,0,c,c) \in \Oh_Y \op \mathcal{L} \op U_1 \op U_1$ and
 use the product $m$ of $\Oh_X$. In order to simplify the notation, we write $(0,0,c,c) = c+\tau^*(c) \in \Oh_X$ and denote 
the product $m$ simply by $\cdot$. Then we have the following expression:
$$
(c+\tau^*(c))^3 = c^3 + \tau^*(c)^3 + 2c\tau^* (c) \left( c+\tau^*(c) \right) \, .
$$
Since $c^3 + \tau^*(c)^3 , 2c\tau^* (c) \in \Oh_Y$, the minimal cubic polynomial of $c+\tau^*(c)$
has no square term and so it belongs to the Tschirnhausen module of $\Oh_W$ over $\Oh_Y$.
This proves the first claim.

For the second claim, recall that by definition $\phi_2$ is the composition of the product 
in $\Oh_W$ followed by the projection onto $U_1$. For any $c, c' \in U_1$,  
the product between $c+\tau^*(c), c'+\tau^*(c') \in \Oh_W$ has the following expression: 
$$
\left(c+\tau^*(c) \right)\left(c'+\tau^*(c')\right) = cc' +\tau^*(cc') + c\tau^*(c') + \tau^*(c\tau^*(c')) \, .
$$ 
Notice that $c\tau^*(c') + \tau^*(c\tau^*(c')) \in \Oh_Y$ and that $cc' +\tau^*(cc')$
corresponds to $(0,\tau^*(cc')) \in U_1$ under the identification \eqref{OW}.
So $\phi_2$ coincides with $m_{11}$ under the identification $\tau^* \colon U_1 \to U_2$.
\end{proof}

Let $f\colon W \to Y$ be a flat finite map of degree $3$, let $E$ be the Tschirnhausen module
of $\Oh_W$ over $\Oh_Y$ and let $\phi_2 \colon {\rm Sym}^2 (E) \to E$ be the 
associated triple cover homomorphism, as defined in \cite{miranda}. 
Under the standard identification $E \cong E^\vee \ot \we^2 (E) = {\rm Hom} (E, \we^2 (E))$, $u\mapsto (v\mapsto u\we v)$,
we can view  $\phi_2$ as a morphism $\phi_2 \colon {\rm Sym}^2 (E) \ot E \to \we^2 (E)$.
By \cite[Prop. 3.5]{miranda}, $\phi_2$ is symmetric, hence it induces a morphism
$\Phi \colon {\rm Sym}^3 (E) \to \we^2 (E)$. The structure theorem for triple
covers states that conversely $f\colon W \to Y$ is determined by 
a rank $2$ locally free sheaf $E$ of $\Oh_Y$-modules   and 
an $\Oh_Y$-morphism $\Phi \colon {\rm Sym}^3 (E) \to \we^2 (E)$ (\cite[Thm. 3.6]{miranda}).
The proof of the fact that  $\phi_2$ is symmetric given in \cite{miranda}  is done by a direct calculation using
local coordinates and is a consequence of the fact that $E$ is the Tschirnhausen module of $\Oh_W$
over $\Oh_Y$. 

We  extend  part of these results to  $D_n$-covers, $n\geq 2$,
and  give a different proof of \cite[Prop. 3.5]{miranda} quoted above.
For any $D_n$-cover $\pi \colon X \to Y$, there is a $\Oh_Y$-morphism 
$$
\phi_{n-1}\colon {\rm Sym}^{n-1}(U_1) \to U_1
$$
defined as follows: for any $u_1\ot \ldots \ot u_{n-1} \in {\rm Sym}^{n-1}(U_1)$, their product  
$m \left( u_1\ot \ldots \ot u_{n-1} \right)$ is in $U_{n-1}$, so
applying $\tau^* \colon U_{n-1} \to U_1$ we obtain an element in $U_1$, this is
by definition $\phi_{n-1}(u_1\ot \ldots \ot u_{n-1})$. 
Under the identification $U_1 \cong U_1^\vee \ot \we^2(U_1)$
as before, we can see $\phi_{n-1}$ as a morphism 
$$
\phi_{n-1} \colon {\rm Sym}^{n-1}(U_1) \ot U_1 \to \we^2(U_1) \, .
$$
\begin{prop}
The morphism $\phi_{n-1}$ is symmetric, hence induces a morphism of $\Oh_Y$-modules
$\Phi_{n-1} \colon {\rm Sym}^{n}(U_1)  \to \we^2(U_1)$.
\end{prop}
\begin{proof}
By the commutativity of $m$, $\Phi_{n-1}$ is symmetric in the first $n-1$ entries. 
So, it is enough to prove that 
\begin{equation}\label{phin-1sym}
\Phi_{n-1}(u_1\ot \ldots \ot u_{n-1}\ot v) = \Phi_{n-1}(u_1\ot \ldots \ot u_{n-2}\ot v \ot u_{n-1}) \, ,
\end{equation}
for any $u_1, \ldots , u_{n-1} , v \in U_1$.
To this aim, let us consider the morphism 
$$
m_{1,n-1}^- \colon \we^2 (U_1) \to \mathcal{L}
$$
defined in  \eqref{min-ipm}. By Proposition \ref{U3},
$m_{1,n-1}^-$ identifies $\we^2 (U_1)$ with $\mathcal{L}(-\D_p)$, hence 
\eqref{phin-1sym} is equivalent to 
\begin{eqnarray}\label{phin-1sym+}
&&m_{1,n-1}^-  \left( \Phi_{n-1}(u_1\ot \ldots \ot u_{n-1}\ot v)\right) =  \nonumber \\
&&=m_{1,n-1}^- \left(  \Phi_{n-1}(u_1\ot \ldots \ot u_{n-2}\ot v \ot u_{n-1}) \right) \, .
\end{eqnarray}

From the explicit form of the isomorphism $U_1 \cong U_1^\vee \ot \we^2(U_1)$
recalled before, we have that 
$$
\Phi_{n-1}(u_1\ot \ldots \ot u_{n-1}\ot v) = \tau^*(m(u_1\ot \ldots \ot u_{n-1})) \we v \, .
$$ 
Furthermore, by definition, $m_{1,n-1}^- \left( \tau^*(m(u_1\ot \ldots \ot u_{n-1})) \we v \right)$
is the projection onto $\mathcal{L}$ of $m\left( \tau^*(m(u_1\ot \ldots \ot u_{n-1})) \ot \tau^*(v) \right)
\in \Oh_Y \op \mathcal{L}$. Since $\tau^*$ commutes with $m$, 
$$
m\left( \tau^*(m(u_1\ot \ldots \ot u_{n-1})) \ot \tau^*(v) \right) = \tau^* m\left( m(u_1\ot \ldots \ot u_{n-1})) \ot v \right) \, .
$$
Using the associativity and the commutativity of $m$, we deduce that
$$
\tau^* m\left( m(u_1\ot \ldots \ot u_{n-1})) \ot v \right) = \tau^* m\left( m(u_1\ot \ldots \ot u_{n-2} \ot v)) \ot u_{n-1} \right) \, .
$$ 
Projecting both sides of this equation onto $\mathcal{L}$ we obtain \eqref{phin-1sym+}, 
from which the statement follows.
\end{proof}

\begin{rem}
From Theorem \ref{STDn} it follows that, for a $D_n$-cover $\pi \colon X \to Y$,
the intermediate degree $n$ cover $f\colon W:= X/\langle \tau \rangle \to Y$
embeds  naturally in $\Oh_Y  \op U_1 \op \ldots \op U_{\frac{n-1}{2}}$
if $n$ is odd, respectively in  $\Oh_Y  \op U_1 \op \ldots \op U_{\frac{n}{2}-1} \op \mathcal{M}$
when $n$ is even. Then the symmetric
morphism $\Phi_{n-1} \colon {\rm Sym}^n(U_1) \to \we^2(U_1)$
determines $f\colon W \to Y$ birationally.
\end{rem}

\section{Simple and almost simple dihedral covers and their invariants}

In this section we define  the  \textit{simple dihedral covers}
and the \textit{almost simple dihedral covers},
and we investigate some of their properties. Throughout the section we assume that $Y$ is
a smooth variety.

Recall that, from Proposition \ref{dfe}, a $D_n$-field extension $\CC ( Y ) \subset E$ is given by
$E = \CC ( Y ) ( x )$, where $x$ satisfies an irreducible equation of the following form:
$$
x^{2n}-2ax^n + F^n = 0 \, ,
$$
with $a, F \in \CC ( Y )$. 
In order to construct varieties having $E$ as field of rational functions, one could proceed in the following way.
Consider a geometric line bundle $\mathbb{L} \to Y$, sections $a\in H^0 (Y, \LL^{\ot n})$,
$F\in H^0 (Y, \LL^{\ot 2})$, and define 
$$
X' := \{ v\in \LL \, | \, v^{2n} -2a v^n + F^n =0 \} \subset \LL \, ,
$$
where $v$ is a fibre coordinate of $\LL$. The singular locus of
$X'$ contains the locus $v=F=0$, so it has codimension $1$ and $X'$ is not normal. 

We slightly modify this construction in the following way.
Let $\LL \to Y$  be a geometric line bundle, 
$a\in H^0 (Y, \LL^{\ot n})$ and  $F\in H^0 (Y, \LL^{\ot 2})$.
Define $X \subset \LL \op \LL$ to be the set of $(u,v) \in \LL \op \LL$ such that the following 
equations are satisfied:
\begin{eqnarray}\label{sde}
\begin{cases}
uv & = F \\
u^n -2a + v^n & = 0 \, ,
\end{cases}
\end{eqnarray}
where $u$ (resp. $v$) is a fibre coordinate of the first (resp. second) copy of $\LL$ in $\LL \op \LL$. 
The dihedral group $D_n$ acts on $X$ via $\s (u,v) = (\z u , \z^{-1} v)$, $\tau (u,v) =(v,u)$,
where $\z \in \CC^*$ is a primitive root of $1$. 

\begin{theo}\label{sdn}
Let $Y$ be a smooth variety. Let $\LL \to Y$ be a geometric line bundle, $a\in H^0 (Y, \LL^{\ot n})$ and
$F\in H^0 (Y, \LL^{\ot 2})$, such that:
\begin{itemize}
\item[(i)]
the zero locus of $ a^2 - F^n \in H^0 (Y, \LL^{\ot 2n})$ is smooth in the open set $F\not= 0$;
\item[(ii)]
the divisors $\{ a=0\}$ and $\{F=0\}$ intersect each other transversely. 
\end{itemize}
Then the variety $X$ defined by the 
equations \eqref{sde} is smooth and the restriction
to $X$ of the fibre bundle projection $\LL \op \LL \to Y$  is a $D_n$-cover, $\pi \colon X \to Y$,
with branch divisor $\sB_\pi=\{ F^n - a^2 =0 \}$.
Furthermore, if $\{ a=0 \} \cap \{ F=0 \} \not= \emptyset$, then $X$ is irreducible.
\end{theo}

\begin{defin}
A {\bf simple $D_n$-cover} of $Y$ is the $D_n$-cover $\pi \colon X \to Y$
given as in Theorem \ref{sdn} by the restriction to $X$ of the fibre bundle projection $\LL \op \LL \to Y$.
\end{defin}

\begin{proof}(Of Theorem \ref{sdn}.)
Let us define $\Phi_1 :=  uv - F$ and  $\Phi_2 :=  u^n -2a + v^n$, so that the equations \eqref{sde} become
\begin{eqnarray*}
\begin{cases}
\Phi_1 &= 0  \\
\Phi_2 &=  0  \, .
\end{cases}
\end{eqnarray*}
Taking partial derivatives along the fibre coordinates we have:
$$
\left(
\begin{matrix}
\frac{\partial \Phi_1}{\partial u} & \frac{\partial \Phi_1}{\partial v} \\
& \\
\frac{\partial \Phi_2}{\partial u} & \frac{\partial \Phi_2}{\partial v}
\end{matrix}
\right) =
\left(
\begin{matrix}
v & u \\
& \\
nu^{n-1} & nv^{n-1}
\end{matrix}
\right) \, ,
$$
hence the ramification divisor of $\pi$ is $R=\{v^n - u^n =0\}$. Notice that 
$v^n - u^n$ is $\langle \s^* \rangle$-invariant, while $\tau^*(v^n - u^n)=-(v^n - u^n)$.
Therefore the branch locus $\sB_\pi$ is defined by the following equation:
\begin{eqnarray*}
(v^n - u^n)(u^n-v^n) &=& - (v^n - u^n)^2 \\
& = & - (v^n + u^n)^2 + 4 (uv)^n \\
& = & -4a^2 +4F^n \, .
\end{eqnarray*}

Let us consider first the restriction of $\pi$ over the points where $F\not= 0$.
In this locus, we have  $v= F/u$ and  so the equation  $\Phi_2 =0$ is equivalent to
$$
\Phi =0 \, , \quad \mbox{where} \quad \Phi: = u^{2n} -2au^n + F^n  \, .
$$
It follows  that $\pi$ is finite over the open set $F\not= 0$. To show the flatness of $\pi$
over $F \not= 0$, by Prop. \ref{fcm} it suffices to prove that $X$ is smooth there. To this aim consider
$$
\frac{\partial \Phi}{\partial u} =  2n u^{n-1} ( u^n -a ) \, ,
$$
which can vanish only if $u^n -a =0$, since $F\not= 0 \Rightarrow u\not= 0$.
On the other hand, over the locus $u^n -a=0$, 
\begin{eqnarray*}
\frac{\partial \Phi}{\partial y} & = & -2 u^n \frac{\partial a}{\partial y} + \frac{\partial F^n}{\partial y}  \\
&=& \frac{\partial (F^n -a^2)}{\partial y} \, ,
\end{eqnarray*}
where $y$ is any coordinate function on $Y$. Since $F^n -a^2=0$ is the branch divisor and, by hypothesis (i), 
it is smooth in $F\not= 0$, it follows that  $X$ is smooth there. 

We consider now the restriction of $\pi$ over the locus $F=0$ and
notice that here $uv=0$. If $u\not= 0$, then as before we see that $\pi$
is finite. The smoothness of $X$ at $(u\not=0 , v=0)$ follows from the fact that 
 $(u\not=0 , v=0) \not\in R$. The same argument
applies if $v\not= 0$. It remains the case where $u=v=0$, which implies
that $a=F=0$. By hypothesis (ii), $a$ and $F$ are part of a local coordinate system
near the points where $a=F=0$, in the analytic topology. Writing the equations for $X$ in these coordinates,
we see that $X$ is smooth at these points. There remains to show that $\pi$ is finite also
over the points $a=F=0$. Here $uv$ and $u^n +v^n$ are precisely the $D_n$-invariants,
and $\Oh_X$ is a free $\Oh_Y$-module generated by 
$1, u^n - v^n , u, u^2, \ldots , u^{n-1} , v, v^2 , \ldots , v^{n-1} $.

Let us now assume that $\{ a=0 \} \cap \{ F=0 \} \not= \emptyset$.
The irreducibility of $X$ is equivalent to the surjectivity of the monodromy
of the cover, $\mu \colon \pi_1 (Y\setminus \sB_\pi) \to D_n$.
Since the branch divisor of the intermediate cover $q\colon Z \to Y$ coincides with
$\sB_\pi$, it is reduced,  so $Z$ is irreducible and hence $\operatorname{Im} (\mu)$
contains a reflection $\s^i \tau$. To see that also $\s \in \operatorname{Im} (\mu )$, 
let us  choose local  coordinates $(a,F, y_3, \ldots , y_{\dim (Y )})$
for $Y$ at a point in $\{ a=0 \} \cap \{ F=0 \}$ and  
consider the path in $X$
$$
\gamma (t ) = \left(\exp \left( \frac{2\pi \sqrt{-1}}{n}t \right)u_0,
\exp \left( - \frac{2\pi \sqrt{-1}}{n}t\right)v_0, F, a, y_3 , \ldots , y_{\dim (Y )}
\right) \, ,
$$
where, $(u_0 , v_0 , F, a, y_3 , \ldots , y_{\dim (Y)}) \in X$.
Then, $\gamma ( 1) = \s \gamma ( 0 )$, and $\pi \circ \gamma $
is a loop contained in $Y\setminus \sB_\pi$, if and only if 
$$
\exp ( {2\pi \sqrt{-1}}t )u_0^n - \exp ( - {2\pi \sqrt{-1}}t)v_0^n \not= 0 \, , \quad \forall \, t \, .
$$
Since this condition can be easily achieved, e.g. by choosing 
$u_0, v_0$ with $|u_0 | \not= | v_0|$, the claim follows.
\end{proof}

Notice that the hypothesis (i) in the previous theorem is general, as it follows from 
the next proposition.
\begin{prop}
Under the same notation as before, for a general choice of  $a\in H^0 (Y, \LL^{\ot n})$ and
$F\in H^0 (Y, \LL^{\ot 2})$, $F^n -a^2 =0$ is smooth over $F\not= 0$.
\end{prop}
\begin{proof}
This follows directly from Bertini's theorem (see e.g. \cite{GH}).
\end{proof}

\subsection{Invariants of simple dihedral covers} 

We first determine  the eigensheaves decomposition of $\pi_* \Oh_X$, where 
$\pi \colon X \to Y$ is a simple $D_n$-cover. Let $\LL , F$ and $a$ be as in the statement of 
Thm. \ref{sdn}, and let us denote with $\Oh_Y ( L )$ the sheaf of sections of $\LL$.
Then, over an open subset $S\subset Y$ where $\LL$ is trivial, from \eqref{sde} we deduce that
\begin{eqnarray*}
{\pi_* \Oh_X}_{|S} &\cong& \frac{\Oh_S [\lambda ,\mu]}{\left( \lambda \mu - F , \lambda^n -2a + \mu^n \right)} \\
&\cong& \Oh_S  \op \Oh_S (\la^n - \mu^n ) \op_{i=1}^{n-1} \left(  \Oh_S \la^i \op \Oh_S \mu^{ n-i} \right)  \, ,
\end{eqnarray*}
where $\la$ and $\mu$ are fibre coordinates on the dual $\LL^\vee$
of $\LL$. This local description  globalizes and we 
have:
\begin{equation}\label{pi*ox}
{\pi_* \Oh_X} \cong \op_{i=0}^{n-1} \big[ \Oh_Y (-iL) \op \Oh_Y (-(n-i)L) \big] \, .
\end{equation}
Notice that, for the sheaves $\sL$ and $U_i$ introduced in Section 5.2, we have:
$$
\sL = \Oh_Y (-nL) \, , \quad U_i = \Oh_Y (-iL) \op \Oh_Y (-(n-i)L) \, .
$$

To determine the canonical bundle of $X$, we use the canonical bundle formula for branched covers:
$$
\om_X = \pi^* \om_Y \ot \Oh_X ( R ) \, ,
$$
where $R$ is the ramification divisor of $\pi$. If $\pi \colon X \to Y$ is a simple $D_n$-cover,
then $R= \{ u^n - v^n = 0 \}$ (see the proof of Thm. \ref{sdn}). 
Since $u^n - v^n$ is a generator of the eigensheaf corresponding to the irreducible  representation 
of $D_n$ with character $\chi_2$, $\Oh_X ( R ) = \pi^* \mathcal{L}^\vee = \pi^* \Oh_Y (nL)$.
Hence
\begin{eqnarray}\label{omegax}
\om_X = \pi^* \left( \om_Y (nL) \right) \, .
\end{eqnarray}
In particular, if $\dim (X ) = 2$, then the self-intersection of a canonical divisor of $X$ is:
\begin{equation}\label{K2}
K_X^2 = 2n \left( K_Y + nL \right)^2 \, .
\end{equation}

To compute the Euler characteristic $\chi ( \Oh_X )$, by the finiteness 
of $\pi$  we have that 
$\chi ( \Oh_X ) = \chi ( \pi_* \Oh_X )$. In particular, if $\dim (X)=2$, 
the Riemann-Roch theorem for surfaces yields the following formula:
\begin{equation}\label{chiox2}
\chi (\Oh_X ) = 2n \chi (\Oh_Y) + \frac{1}{6}n(2n^2 +1) L \cdot L +
\frac{1}{2}n^2 L \cdot K_Y \, .
\end{equation}

\subsection{Almost simple dihedral covers}
In this section we define the {almost simple dihedral covers},
which can be seen as projectivizations of the simple dihedral covers. 
The construction follows closely the one of 
almost simple cyclic covers, introduced and studied in \cite{ENR}.

{
The covering space $X$
of an almost simple dihedral cover $\pi \colon X \to Y$, over the smooth variety $Y$, 
is defined as a complete 
intersection in the fibre product $\PP_1 \times_Y \PP_2 \to Y$
of two $\PP^1$-bundles over $Y$ in the following way.
Let $\LL \to Y$ be a geometric line bundle
 and
$\ul{\CC}=Y\times \CC \to Y$ be the trivial geometric line bundle.
For each $i=1,2$, 
let $\PP_i := \PP(\ul{\CC} \op \LL)$ be  the  $\PP^1$-bundle associated to  
$\ul{\CC} \op \LL \to Y$, 
and let $p_i \colon \PP_i \to Y$ be the 
corresponding projection.  Here, each fibre of $p_i$ is the projective space 
of $1$-dimensional subspaces in the corresponding fibre of $\ul{\CC} \op \LL$,
hence, using the notation in \cite{Hart},  $\PP_i = 
{\rm{Proj}}\left( \Oh_Y \op \Oh_Y(-L)\right)$ and in particular 
$(p_i)_* \Oh_{\PP_i}(1) = \Oh_Y \op \Oh_Y(-L)$,
where $L$ is a divisor in $Y$ such that $\Oh_Y (L)$ 
is the sheaf of sections of $\LL$ (hence $\LL = {\rm Spec}({\rm Sym}(\Oh_Y(-L)))$).
On each $\PP_i$ there  are projective fibre coordinates:
$[u_0: u_1]$ for $\PP_1$, $[v_0: v_1]$ for $\PP_2$.
$u_0, u_1$ are defined as follows ($v_0,v_1$ are defined in the same way): 
let $U\subset p_1^*(\ul{\CC} \op \LL)$
be the universal sub-bundle, then $u_0 \colon U \to p_1^*(\ul{\CC})$
is the composition of the inclusion with the projection; similarly for
$u_1 \colon U \to p_1^*\LL$. Notice that, since  $U= {\rm Spec}({\rm Sym}\, \Oh_{\PP_1}(1))$,
then $u_0 \in H^0(\PP_1, \Oh_{\PP_1}(1))$ and 
$u_1 \in H^0(\PP_1, \Oh_{\PP_1}(1)\ot p_1^*\Oh_Y(L))$.
}

Let 
$F\in H^0(Y, \LL^{\ot 2})$ and
let $A_0 , A_\infty$ be effective divisors 
in $Y$ such that
\begin{equation*}
A_0 \equiv nL + A_\infty \, .
\end{equation*}
Let $a_\infty \in H^0(Y,\Oh_Y(A_\infty))$ and $a_0 \in H^0(Y,\Oh_Y(A_0))$ be
such that $A_\infty = \{ a_\infty =0 \}$ and $A_0= \{ a_0 = 0 \}$.

Consider the subvariety $X\subset \PP_1 \times_Y \PP_2= \PP ( \ul{\CC} \op \LL) \times_Y \PP ( \ul{\CC} \op \LL)$ 
defined by the following equations:
\begin{eqnarray}\label{asde}
\begin{cases}
\Phi_1 & = 0 \\
\Phi_2 & = 0 \, ,
\end{cases}
\end{eqnarray}
where 
$$
\Phi_1:= u_1 v_1 - u_0 v_0 F \, , \quad
\Phi_2:=a_\infty v_1^n u_0^n -2a_0 v_0^n u_0^n + a_\infty v_0^n u_1^n \, ,
$$
 $u_0, u_1, v_0, v_1$ are  defined above and
$( [u_0 : u_1], [v_0 : v_1])$ are projective fibre coordinates on
$\PP ( \ul{\CC} \op \LL) \times_Y \PP ( \ul{\CC} \op \LL)$.

Notice that the dihedral group $D_n$ acts on $X$ in the following way: 
\begin{eqnarray*}
\s ( [u_0 : u_1], [v_0 : v_1]) &=& ([u_0 : \z u_1], [v_0 : \z^{-1} v_1]) \, ,\\
\tau ([u_0 : u_1], [v_0 : v_1]) &=& ([v_0 : v_1], [u_0 : u_1]) \, ,
\end{eqnarray*}
where $\z \in \CC^*$ is a primitive $n$-th root of $1$. 
Then we have the following result.
\begin{theo}\label{asdn}
Let $Y$ be a smooth variety, $\LL \to Y$ be a geometric line bundle,
$\ul{\CC}=Y\times \CC \to Y$ be the trivial geometric line bundle, and 
$F\in H^0(Y, \LL^{\ot 2})$. 
Let $A_0 , A_\infty$ be effective divisors 
in $Y$ such that
\begin{equation*}
A_0 \equiv nL + A_\infty \, ,
\end{equation*}
where $L$ is a divisor in $Y$ with  $\Oh_Y (L)$ being the sheaf of sections of $\LL$.
Let $a_\infty \in H^0(Y,\Oh_Y(A_\infty))$ and $a_0 \in H^0(Y,\Oh_Y(A_0))$ be
such that $A_\infty = \{ a_\infty =0 \}$ and $A_0= \{ a_0 = 0 \}$.

Assume that the following conditions are satisfied:
\begin{itemize}
\item[(i)]
$A_0$ intersects $\{ F=0\}$ transversely, $A_0 \cap A_\infty = \emptyset$ and
$A_\infty$ is smooth;
\item[(ii)]
the locus $\{ a_0^2 - F^n a_\infty^2 =0\}$ is smooth on the open set $F\not= 0$.
\end{itemize}
Then $X$, defined by \eqref{asde}, is smooth and the restriction of the projection 
$\PP (\ul{\CC} \op \LL) \times_Y \PP (\ul{\CC} \op \LL)\to Y$ to $X$ is a $D_n$-cover 
$\pi \colon X \to Y$ with branch divisor 
$\sB_\pi = \{ a_\infty (a_0^2 - a_\infty^2 F^n)=0\}$.
Furthermore, if $A_0 \cap \{ F=0 \} \not= \emptyset$, then $X$ is irreducible.
\end{theo}

\begin{defin}
An {\bf almost simple $D_n$-cover} of $Y$ is the $D_n$-cover $\pi \colon X \to Y$
given as in Theorem \ref{asdn} by the restriction to $X$ of the  
fibre bundle projection $\PP (\ul{\CC} \op \LL) \times_Y \PP (\ul{\CC} \op \LL)\to Y$.
\end{defin}

\begin{proof} 
We first prove that the intersection of $X$ with each one  of the  standard open subsets
$v_0 u_0 \not= 0$, $v_1 u_0 \not= 0$ and $v_1 u_1 \not= 0$ is a smooth variety. 
This suffices because $\tau (\{ v_0 u_1 \not= 0 \} ) = \{v_1 u_0 \not= 0\}$.

To this aim, observe  that the restriction of $\pi$ to the locus where $u_0v_0 \not= 0$
is a simple $D_n$-cover. Indeed, setting $u=u_1/u_0$ and $v=v_1/v_0$, 
the equations \eqref{asde} reduce to 
\begin{eqnarray*}
\begin{cases}
u v -  F & = 0 \\
a_\infty v^n  -2a_0  + a_\infty  u^n & = 0 \, .
\end{cases}
\end{eqnarray*}
Since $A_\infty \cap A_0 = \emptyset$, on this locus $a_\infty$ never vanishes,
and setting $a=a_0/a_\infty$ we obtain the equations \eqref{sde}. Under our hypotheses Theorem 
\ref{sdn} applies, hence $X$ is smooth if $u_0v_0 \not= 0$.

Consider now the locus $v_1 u_1 \not= 0$ and observe that there $F$ never vanishes. 
Let us define $u=u_0/u_1$, $v=v_0/v_1$ and
$g=1/F$. Then the equations \eqref{asde} become
\begin{eqnarray*}
\begin{cases}
u v  & = g \\
a_\infty u^n  -2a_0 g^n  + a_\infty  v^n & = 0 \, .
\end{cases}
\end{eqnarray*}
Since $A_\infty \cap A_0 = \emptyset$, $a_\infty$ never vanishes in this locus.
So, defining   $a=a_0/a_\infty$, we get the following equations for $X$: 
$$
u^{2n} -2ag^n u^n + g^n =0 \, ,  \quad v=g/u \, .
$$
Notice that this is the equation of a $D_n$-cover, with action $\s (u) = \z u$, $\tau (u) = g/u$.
Since $g\not= 0$ everywhere, $\langle \s \rangle \cong \ZZ/n\ZZ$ acts freely
and so $X$ is smooth if and only if the intermediate double cover is smooth.
The intermediate double cover has equation $z^2 -2ag^n z +g^n =0$ and its
branch divisor is $\{g^n( g^na^2 - 1) =0\} = \{ a^2 - F^n =0\}$. 
By hypothesis (ii) this locus is smooth, so $X$ is smooth where $v_1 u_1 \not= 0$.

It remains to consider the case where $v_1u_0 \not= 0$.  Here, setting $v_1 = 1 = u_0$, the equations \eqref{asde} reduce to
\begin{eqnarray*}
\begin{cases}
u_1  -  v_0 F & = 0 \\
a_\infty   -2a_0 v_0^n  + a_\infty v_0^n u_1^n & = 0 \, .
\end{cases}
\end{eqnarray*}
Substituting $u_1  =  v_0 F$ in the second equation we obtain: 
$$
a_\infty   -2a_0 v_0^n  + a_\infty F^n v_0^{2n}  = 0 \, .
$$
Notice that, if $v_0 \not= 0$, then we have already seen that $X$ is smooth.
On the other hand, when $v_0 = 0$, we must have $a_\infty =0$, and then the smoothness of $A_\infty$
implies that of $X$.

The finiteness of $\pi$ follows from the fact that $\pi$ has finite fibre and
its restriction to each open subset of $Y$ where $\LL$ is trivial
is projective (\cite{Hart}, III, Ex. 11.1).

To describe the branch divisor, recall that the restriction of $\pi$ to the open set $u_0v_0 \not= 0$
is a simple dihedral cover and so its branch divisor is  $F^n - a^2 =0$, where 
 $a:=a_0/a_\infty$. Notice that, if $u_0v_0 \not= 0$, then $a_\infty \not= 0$
everywhere, on the other hand, if $u_0v_0 = 0$, then $u_1v_1 =0$ by \eqref{asde}, so
$X\cap \{ u_0v_0 = 0 \} \subset \{u_1v_0\not= 0 \} \cup \{u_0v_1\not= 0 \}$.
The claim now follows from the previous explicit description of $\pi$ on $u_0v_1\not= 0$.

Finally, if $A_0 \cap \{ F=0 \} \not= \emptyset$, then $X$ is irreducible 
since the open subset $X\cap \{ u_0v_0 \not= 0\}$ is irreducible by Theorem 
\ref{sdn}.
\end{proof}

{
The invariants of almost simple $D_n$-covers $\pi \colon X \to Y$
can be computed in the same way as in the simple case,
once  a description of $\pi_* \Oh_X$ in terms of   $L$ and $A_\infty$ is provided.
In the remaining part of this section we show
that there is an isomorphism as  follows:
\begin{eqnarray}\label{p*oxas}
\pi_* \Oh_X &\cong& \Oh_Y \op \Oh_Y(-nL -A_\infty) \op \\
&&\left( \op_{i=1}^{n-1} [ \Oh_Y(-iL) \op \Oh_Y(-(n-i)L)]\right)(-A_\infty)  \, . \nonumber
\end{eqnarray}
Notice that, on the open subset $Y\setminus A_\infty$ where $\pi$ is a simple cover, 
the previous formula reduces to \eqref{pi*ox}.

To prove \eqref{p*oxas}, we consider the following Cartesian diagram 
$$
\begin{CD}
\sQ @>{\tilde{p}_2}>> \PP_2 \\
@V{\tilde{p}_1}VV @VV{p_2}V \\
\PP_1 @>{p_1}>> Y 
\end{CD} 
$$
where $\sQ:=\PP_1 \times_Y \PP_2$ is a $\PP^1 \times \PP^1$-bundle
with projection $p\colon \sQ \to Y$, $p=p_i \circ \tilde{p}_i$, $\forall \, i=1,2$.
Recall that the Picard group of $\sQ$ is isomorphic to ${\rm Pic}(Y) \times \ZZ^{\oplus 2}$
via the usual isomorphism that sends $(\sL , m, n )\in {\rm Pic}(Y) \times \ZZ^{\oplus 2}$ 
to $p^*\sL \ot \Oh_{\sQ}(m,n):= p^*\sL \ot \Oh_{\PP_1}(m) \boxtimes \Oh_{\PP_2}(n) $.

Let us define $D_i:= \{ \Phi_i = 0 \}$ to be the divisor of $\sQ$ given by the equation $\Phi_i=0$
in  \eqref{asde}, for $i=1,2$, and  notice that 
\begin{eqnarray}\label{Phi12}
&&\Phi_1 \in H^0 \left( \sQ, \Oh_{\sQ}(1,1) \ot p^* \Oh_Y(2L) \right) \, ,
 \nonumber \\
&&\Phi_2 \in H^0 \left( \sQ, \Oh_{\sQ}(n,n) \ot p^* \Oh_Y(nL +A_\infty) \right) \, .
\end{eqnarray}

Then we consider the usual short exact sequence 
\begin{equation}\label{IX}
0\to \sI_X \to \Oh_\sQ \to \Oh_X \to 0 \, ,
\end{equation}
where $\sI_X$ is the sheaf of ideals of $X$.
Since $X$ is the complete intersection of two divisors in $\sQ$,
the Koszul resolution of $\sI_X$ is as follows:
\begin{equation}\label{Koszul}
0\to \Oh_\sQ (-D_1 - D_2) \to \Oh_\sQ(-D_1) \op \Oh_\sQ(-D_2) \to \sI_X \to 0 \, .
\end{equation}

Applying $p_*$ to \eqref{IX} we obtain the following split short exact sequence: 
$$
0\to \Oh_Y \cong p_* \Oh_\sQ \to p_* \Oh_X \to R^1p_* \sI_X \to 0 \, ,
$$
where we have used the fact that $p_* \sI_X = 0 = R^1p_* \Oh_\sQ$.

In order to compute $R^1p_* \sI_X$, we apply $p_*$ to \eqref{Koszul}
and we obtain the following  exact sequence:
$$
0\to R^1p_* \sI_X \to R^2 p_* \Oh_\sQ (-D_1 - D_2) \to R^2 p_* 
\left( \Oh_\sQ(-D_1) \op \Oh_\sQ(-D_2) \right) \, ,
$$
where we have used the equality $ R^1 p_* 
\left( \Oh_\sQ(-D_1) \op \Oh_\sQ(-D_2) \right)=0$ that follows from 
the K\"unneth formula. Furthermore, since 
$ R^2 p_* \Oh_\sQ(-D_1) =0$, we have that
$$
R^1p_* \sI_X = \ker \big[ R^2 p_* \Oh_\sQ (-D_1 - D_2) \stackrel{\mu}{\to} R^2 p_* 
 \Oh_\sQ(-D_2) \big] \, ,
$$
where $\mu$ is induced by the morphism $\Oh_\sQ (-D_1 - D_2) \to  \Oh_\sQ (-D_2 )$
in the Koszul resolution of $\sI_X$,  given by $\psi_1 \we \psi_2 \mapsto \psi_1 (\Phi_1) \psi_2$,
for $\psi_i \in \Oh_\sQ (-D_i)$.

To describe $\mu$ explicitly, we use the following isomorphisms:
\begin{eqnarray*}
R^2 p_* \Oh_\sQ (-D_1 - D_2) &\cong& S^{n-1} \left( \Oh_Y \op \Oh_Y(L) \right)^{\ot 2} \ot \Oh_Y (-nL - A_\infty) \, , \\
R^2 p_* \Oh_\sQ (- D_2) &\cong& S^{n-2} \left( \Oh_Y \op \Oh_Y(L) \right)^{\ot 2} \ot \Oh_Y (-(n-2)L - A_\infty) \, ,
\end{eqnarray*}
that  follow applying the projection formula, K\"unneth formula, and the standard isomorphisms (see e.g. \cite[Ex. 8.4, III]{Hart}).
Then, if we choose local sections $x_0, y_0$   of $\Oh_Y$, and  $x_1, y_1$   of $\Oh_Y (L)$, that generate the corresponding sheaves,
we obtain the following local  basis  for $S^{n-1} \left( \Oh_Y \op \Oh_Y(L) \right)^{\ot 2}$:
$$
E_{ij}:= x_0^i x_1^{n-1-i} \ot y_0^j y_1^{n-1-j}  \, , \qquad  0\leq i,j \leq n-1 \, .
$$ 
Similarly, $G_{km}:= (x_0^k x_1^{n-2-k} \ot y_0^m y_1^{n-2-m})\ot (x_1\ot y_1)$, for $0\leq k,m \leq n-2$,
is a local basis for $S^{n-2} \left( \Oh_Y \op \Oh_Y(L) \right)^{\ot 2} \ot \Oh_Y (2L)$.
The morphism $\mu$ in these basis is given as follows:
$$
\mu (E_{ij}) = G_{ij} -f G_{i-1, j-1} \, ,
$$
where $F= fx_1 \ot y_1$, and $G_{km}:=0$, for $k,m \not\in \{ 0, \ldots , n-2 \}$.
Hence from elementary linear algebra, we have that $\ker (\mu)$ (twisted by $\Oh_Y(nL + A_\infty)$) is generated by
$E_{n-1,j}$ and $E_{i,n-1}$, for $0\leq i,j \leq n-1$. So
\begin{eqnarray*}
\ker (\mu) &=& \left( [ \op_{j=0}^{n-1} \Oh_Y((n-1-j)L) ] \op [ \op_{i=0}^{n-2} \Oh_Y((n-1-i)L) ] \right) \ot \Oh_Y(-nL-A_\infty) \, ,
\end{eqnarray*}
and hence \eqref{p*oxas} follows.
}

\section{Deformations of simple dihedral covers}

Let $X$ be a simple dihedral covering of $Y$: this means that $X$ is the subvariety of the vector bundle $V = \LL \oplus \LL$
which is (see \ref{sde}) the complete intersection of two hypersurfaces, one in $|p^*(2L)|$, the other in $|p^*(nL)|$;
here $ p \colon V \ra Y$ is the natural projection.

Observe now that the cotangent sheaf of $V$ is an extension of  $p^*  \Om^1_Y $ by
$p^*(\hol_Y (-L)^{\op 2})$,
\begin{equation}\label{cotangentsequence}
0 \to p^*  \Om^1_Y \to \Om_V^1 \to p^*(\hol_Y (-L)^{\op 2}) \to 0 \, ,
\end{equation}
where  $p^*  \Om^1_Y \to \Om_V^1$ is the cotangent map of $p$. 

Then the conormal sheaf exact sequence of $X$ reads out, if we denote by $L' := \pi^* (L)$,  as
$$ 
0 \ra N^*_{X|V} =  \hol_X (-2L' ) \oplus  \hol_X (-nL' ) \ra \Om^1_{V} \ot \Oh_{X} \ra \Om^1_X \ra 0,
$$
and the dual sequence is 
$$ 
0 \ra \Theta_X \ra \Theta_V\ot \Oh_{X} \ra N_{X|V} =  \hol_X (2L' ) \oplus  \hol_X (nL' ) \ra 0,$$ 
whose direct image under $\pi_*$ yields the tangent exact sequence:
$$ 
0 \ra \pi_* \Theta_X \ra \pi_* (\Theta_V\ot \Oh_{X})  \ra  (\hol_Y (2L ) \oplus  \hol_Y (nL))   \otimes \pi_* (\hol_X) \ra 0.$$
Passing to the long exact cohomology sequence, we get the Kodaira -Spencer exact sequence

\begin{equation}\label{KodSpencer}
H^0 ((\hol_Y (2L ) \oplus  \hol_Y (nL))   \otimes \pi_* (\hol_X))  \ra  H^1 (  \Theta_X ) \ra H^1 (\pi_*(\Theta_V\ot \Oh_{X} )) \ra . 
\end{equation}

The meaning of the first linear map is given through the following definition.

\begin{defin}
The space of {\bf natural deformations} of a simple dihedral covering is the family of complete intersections of $V$:
\begin{eqnarray}\label{natdef}
\begin{cases}
uv  - F & = 0 \\
u^n -2a + v^n + \sum_1^{n-1} (b_i u^i + c_i v ^i) + d (u^n - v^n)& = 0 \, ,
\end{cases}\\
{\rm where} \   b_i , c_i \in H^0(\hol_Y(n-i)L), d \in H^0(\hol_Y).
\end{eqnarray}

\end{defin}

\begin{rem}
The reader can see that in  the second equation $$ -2a + v^n + \sum_1^{n-1} (b_i u^i + c_i v ^i) + d (u^n - v^n)$$ 
can be any section in $H^0 ( \hol_Y (nL)   \otimes \pi_* (\hol_X))$, in view of our basic formulae; instead, any section in  
$H^0 ( \hol_Y (2L)   \otimes \pi_* (\hol_X))$ is of the form 
$$  - F + \b u + \a v + \la u^2 + \mu u^2 , \ {\rm where} \  \a, \b \in H^0(\hol_Y(L)), \la, \mu  \in \CC.$$
But the new equation $ uv   - F + \b u + \a v + \la v^2 + \mu u^2 $
is the old form $ u' v' = F'$ (up to a multiplicative constant) if we choose new variables
$$ u' : =  u + \a  + \la  v, \ v'  := v + \b  + \mu  u.$$
\end{rem}

Of course, one can deform not only the equations, but also simultaneously  the base $Y$, the vector bundle $V$, and the equations;
this however leads in general to a deformation with non smooth base. 

We have at any rate an easy result which says that all small deformations are obtained by natural deformations.

\begin{theo}\label{smalldef}
Assume that $\pi \colon X \ra Y$ is a simple dihedral covering. Then all 
small deformations of $X$ are natural deformations of $\pi \colon X \ra Y$,
provided $H^1 (\pi_*(\Theta_V\ot \Oh_{X} ))=0$
(that happens, for example if $ H^1 (( \Theta_Y \oplus \hol_Y (L)^{\op 2}) \otimes \pi_* (\hol_X)  )= 0$). 
In particular, the Kuranishi family of $X$ is smooth, and the Kuranishi space $\operatorname{Def} (X)$  is locally analytically  isomorphic to  
$$  
\operatorname{Def}' : = \coker  \left(  H^0 (\pi_*(\Theta_V\ot \Oh_{X} ) ) \ra  H^0 ((\hol_Y (2L ) \oplus  \hol_Y (nL))\otimes \pi_* (\hol_X) ) \right) \, .
$$
\end{theo}
\proof
By our assumption, and the Kodaira-Spencer exact sequence, we have an isomorphism of vector spaces $\operatorname{Def}' \cong H^1 (\Theta_X)$.
Moreover, by the previous remark, the family of natural deformations has Kodaira-Spencer map which is surjective onto
$\operatorname{Def}$; therefore, by the implicit functions theorem, $\operatorname{Def}(X)$ is the germ of the analytic space $ H^1 (\Theta_X)$ at the origin,
hence our claim.

\qed
\section{Examples and applications}

\subsection{Simple dihedral covers of projective spaces}

We first consider  the case where $Y=\PP^2$, the complex projective plane,
$n=3$ and $\Oh_Y(L) = \Oh_{\PP^2} (1)$. Then 
$a\in H^0(\PP^2, \Oh_{\PP^2}(3))$ and $F\in H^0 (\PP^2, \Oh_{\PP^2}(2))$
are a cubic and a quadric curve respectively, such that the sextic curve  $\{  F^3 - a^2 = 0\}$ is smooth 
in the locus $F\not= 0$, and $\{ a= 0 \} $ intersects $\{ F=0 \}$ transversely. 
By Theorem \ref{sdn}, we have a smooth $D_3$-cover $\pi \colon X \to \PP^2$
branched over $\sB=\{  F^3 - a^2 = 0\}$. Notice  that $\sB\in |-2K_{\PP^2}|$
and $\om_{\PP^2} = \Oh_Y(-3L)$. Hence
$$
\om_X = \pi^* (\om_{\PP^2} (3L) )  \cong \Oh_X \, .
$$
Furthermore,  $q ( X ) = 0$, hence $X$ is a K3 surface.

Let $W= X/ \langle \tau \rangle$ and let $f\colon W \to \PP^2$ be the induced 
triple cover (see Section \ref{d3&triple}). Then $W$ can be realised as a 
cubic surface in $\PP^3$ in such a way that $f$ is the projection from a point 
in $\PP^3 \setminus W$. Indeed, consider the equations \eqref{sde}
which define $X$. If we define $w:= u+ v$, then
\begin{eqnarray*}
w^3 &=& u^3 + v^3 +3uv(u+v) \\
&=& 2a + 3Fw \, .
\end{eqnarray*}
The branch divisor of $f$ is $\sB = \{F^3 - a^2\}$. So, under the hypotheses of
Theorem \ref{sdn}, $\sB$ is a sextic with $6$ cusps lying on a conic.

The fundamental group of the complement $\PP^2 \setminus \sB$
of a sextic curve $\sB$ as above has been studied in \cite{Zar29}, where in particular
it is proven that $\pi_1 (\PP^2 \setminus \sB)$ is generated by two elements
of order $2$ and $3$ respectively. From this it follows that there exists 
a surjective group homomorphism $\pi_1 (\PP^2 \setminus \sB) \to D_3$,
and hence it follows from the generalised Riemann existence theorem of Grauert and Remmert,
that there exists a $D_3$-cover $\pi \colon X \to \PP^2$.
From Theorem \ref{sdn} we have an explicit construction of such a cover.

Let now $n>3$ and consider the simple $D_n$-cover $\pi \colon X \to \PP^2$
associated to $\Oh_Y(L) = \Oh_{\PP^2} (1)$, 
$a\in H^0(\PP^2, \Oh_{\PP^2}(n))$ and $F\in H^0 (\PP^2, \Oh_{\PP^2}(2))$.
Under the hypotheses of Theorem \ref{sdn}, $X$ is a smooth surface with
\begin{eqnarray*}
\om_X &=& \pi^* \Oh_{\PP^2}(n-3) \, , \\
K_X^2 &=& 2n (n-3)^2 \, , \\
\chi (\Oh_X) &=& \frac{1}{3}n^3 - \frac{3}{2}n^2 + \frac{13}{6}n \, ,
\end{eqnarray*}
in particular $X$ is a surface of general type, which is minimal since $K_X \cdot C >0$
for any curve $C\subset X$.

Finally, let $n=2$, $Y=\PP^2$ and $\Oh_Y(L) = \Oh_{\PP^2} (1)$.
Under the hypotheses of Theorem \ref{sdn}, $X$ is a smooth surface with the following invariants:
\begin{eqnarray*}
\om_X &=& \pi^* \Oh_{\PP^2}(-1) \\
\chi (\Oh_X) &=& 1 \\
q&=&0\\
K_X^2 &=& 4 \\
p_n &=& \dim H^0 (X, \om_X^{\ot n}) =0 \, .
\end{eqnarray*} 
Hence $X$ is a rational, non-minimal surface. Indeed, $X$ is
isomorphic to a del Pezzo surface
of degree $4$ in $\PP^4$, the complete intersection of the quadrics
$uv=F(x_0:x_1:x_2)$, $u^2 + v^2 = 2a(x_0:x_1:x_2)$, where 
$(x_0:x_1:x_2:u:v)$ are now homogeneous coordinates in $\PP^4$.

In a similar way one can construct examples of dihedral covers of $\PP^2$
with $\Oh_Y(L)=\Oh_{\PP^2}(d)$, $d> 1$.

As an application of Theorem \ref{smalldef}, we have that all the small
deformations of a simple $D_n$-cover
$\pi \colon X \to \PP^2$ associated to $\Oh_Y (L) = \Oh (m)$, $m\geq 1$,
are natural deformations, if $(m,n)= (1,2)$, or $m\geq 2$ and any $n\geq 2$. 
This follows directly from Thm. \ref{smalldef}, formula \eqref{pi*ox} and 
the computation of the cohomology of $\Om_{\PP^d}^q (k)$ (\cite[p. 256]{bott57}).

\begin{rem}
Notice that, in general, a simple $D_n$-cover of $\PP^d$,
$\pi \colon X \to \PP^d$,
associated to $\Oh_Y (L) = \Oh (m)$,
$F\in H^0(\PP^d , \Oh (2m))$ and $a\in H^0(\PP^d , \Oh (nm))$,
is isomorphic to a complete intersection  $X'$ in the weighted projective
space $\PP^{d+2}(1, \ldots , 1, m,m)$, where
$$
X' := \{ (x_0 : \ldots : x_d : u : v) \in \PP^{d+2}(1, \ldots ,1, m,m) | uv=F , u^n + v^n =2a \}.
$$
To see this, observe that $X \subset \LL \op \LL$ is the  quotient
of $\tilde{X} \subset (\CC^{d+1} \setminus \{0\})\times \CC^2$
via the linear diagonal action of $\CC^*$  with weights $(1,\ldots, 1, m,m)$, where 
$$
\tilde{X}:=\{ (x_0, \ldots , x_d, u, v) | uv = F , u^n +v^n = 2a \} \, .
$$
The claim now follows since $(\CC^{d+1} \setminus \{0\})\times \CC^2 \subset 
\CC^{d+3} \setminus \{0\}$, and the action of $\CC^*$ on 
$(\CC^{d+1} \setminus \{0\})\times \CC^2$ is the restriction of the linear diagonal 
action on $\CC^{d+3} \setminus \{0\}$ with weights $(1,\ldots, 1, m,m)$.

If $d\geq 3$, then  Theorem \ref{smalldef} implies that all the small deformations of $X$
are natural deformations of the simple $D_n$-cover $\pi \colon X \to \PP^d$.
Indeed, in this case, we have that 
\begin{equation*}
H^1 \left( (\Theta_{\PP^d} \op \Oh_{\PP^d} (m)^{\op 2}) \ot \pi_* \Oh_X \right) = 0 \, ,
\end{equation*}
as it follows from the fact that $H^1 (\PP^d , \Oh_{\PP^d} (k)) =0$, $\forall \, k$,  
and that $H^1 (\PP^d , \Theta_{\PP^d} (k) ) \cong H^{d-1} (\PP^d ,\Om^1_{\PP^d}(-k-d-1))^\vee =0$, $\forall \, k$
(\cite{bott57}).
\end{rem}

\subsection{An application to fundamental groups}
According to the generalized Riemann existence theorem of Grauert and Remmert \cite{GR58},
coverings $\pi \colon X \to Y$ of a normal variety $Y$, of degree $n$, with branch  locus contained in a divisor $\sB \subset Y$,
and with $X$ normal, 
correspond to conjugacy classes of group homomorphisms $\mu \colon \pi_1 (Y\setminus \sB) \to \mathfrak{S}_n$.
In this situation, $X$ is irreducible, if and only if $\operatorname{Im} (\mu)$ is a transitive subgroup of $\mathfrak{S}_n$;
$\pi \colon X \to Y$ is Galois with group $G:= \operatorname{Im} (\mu)$, if and only if $G$ coincides with the group
of automorphisms  of  $\pi $.
In particular, if $\pi \colon X \to Y$ is a $G$-cover with $G$ non-abelian and $X$ irreducible,
then $\pi_1 (Y\setminus \sB)$ is necessarily non-abelian. 

Fundamental groups of complements of divisors in projective varieties have been 
extensively studied by many authors. 
As a direct consequence of  Theorem \ref{sdn}, we have the following 
result.

\begin{prop}
Let $Y$ be a smooth variety and $L \subset Y$ be a divisor. Assume that there exist $a\in H^0 (Y, \Oh_Y ( nL))$ and
$F\in H^0 (Y, \Oh_Y (2L))$, such that the conditions (i), (ii) of Thm. \ref{sdn} are satisfied 
and $\{ a=0 \} \cap \{ F=0 \} \not= \emptyset$. 
Then $\pi_1(Y \setminus \sB)$ admits an epimorphism onto $D_n$, in particular it is non-abelian, where 
$\sB = \{ a^2 - F^n =0 \}$.  
\end{prop}

Notice that similar results have been obtained using different methods (\cite[Lemma 3]{CKO03},
\cite{ArCo10}). Briefly, one considers the pencil $\{ \la 2 ( a  =0) + \mu n( F=0 ) \}_{(\la :\mu) \in \PP^1}
\subset |2nL|$ and the induced  morphism $Y\setminus \sB \to \PP^1 \setminus \{ (1:1) \}$.
This gives a group homomorphism $\pi_1 (Y\setminus \sB) \to \pi_1^{orb}(\PP^1 \setminus \{ (1:1) \})$,
where $\pi_1^{orb}$ is the orbifold fundamental group of $\PP^1 \setminus \{ (1:1) \}$
with two orbifold points, $(1:0)$  of order $2$ and $(0:1)$ of order $n$. Now, 
using the long exact homotopy sequence, one concludes that $\pi_1 (Y\setminus \sB) \to \pi_1^{orb}(\PP^1 \setminus \{ (1:1) \})$
is surjective and so $\pi_1 (Y\setminus \sB)$ is not abelian.

\medskip
\noindent {\bf Authors' Address:}\\
\noindent Fabrizio Catanese, 
\\ Lehrstuhl Mathematik VIII,\\ 
Mathematisches
Institut der Universit\"at
Bayreuth\\ NW II,  Universit\"atsstr. 30\\ 95447 Bayreuth (Germany).\\

\noindent Fabio Perroni, \\
Dipartimento di Matematica e Geoscienze, \\
Sezione di Matematica e Informatica,
 Universit\`a di Trieste, Via Valerio 12/1, 34127 Trieste (Italy).\\
 
\noindent          Emails:        Fabrizio.Catanese@uni-bayreuth.de;
 fperroni@units.it.

\end{document}